\documentclass[10pt,draft,reqno]{amsart}
     \makeatletter
     \def\section{\@startsection{section}{1}%
     \z@{.7\linespacing\@plus\linespacing}{.5\linespacing}%
     {\bfseries
     \centering
     }}
     \def\@secnumfont{\bfseries}
     \makeatother
\newtheorem{teo}{Theorem}[section]
\newtheorem{defi}{Definition}[section]

\newtheorem{lemma}{Lemma}[section]

\numberwithin{equation}{section}
\begin{document}

\title[The Boundedness of General Alternative  Gaussian Singular Integrals ] {The Boundedness of General  Alternative Gaussian Singular Integrals with respect to the Gaussian measure}

\author{Eduard Navas}
\address{Departamento de Matem\'aticas, Universidad Nacional Experimental Francisco de Miranda, Punto Fijo, Venezuela.}
\email{[Eduard Navas]enavas@correo.unefm.edu.ve}
  \author{Ebner Pineda}
\address{Escuela Superior Polit\'ecnica del Litoral. ESPOL, FCNM, Campus Gustavo Galindo Km. 30.5 V\'ia Perimetral, P.O. Box 09-01-5863, Guayaquil, ECUADOR.}
\email{epineda@espol.edu.ec}
\author[Wilfredo~O.~Urbina]{Wilfredo~O.~Urbina*}
\thanks{* Corresponding author}
\address{Department of Mathematics, Actuarial Sciences  and Economy, Roosevelt University, Chicago, IL,
   60605, USA.}
\email{wurbinaromero@roosevelt.edu}



\begin{abstract}
In this paper we introduce a new class of Gaussian singular integrals, the general alternative Gaussian singular integrals and study the boundedness of them in  $L^p(\gamma_d)$, $ 1 < p < \infty$ and its weak $(1,1)$ boundedness with respect to the Gaussian measure following \cite{pe}  and \cite{aimarforzaniscot}, respectively.
\end{abstract}

\maketitle

\section{Introduction and Preliminaries}

Singular integrals are some of the most important operators in classical harmonic analysis.
They first appear naturally in the proof of the $L^p(\mathbb{T}) $ convergence of Fourier series, $1 < p <\infty$; where the notion  of the {\em conjugated function} is needed\footnote{For  a detailed study of this problem see for instance E. Stein \cite[Chapter II, III]{st1}, J. Duoandikoetxea \cite [Chapter 4, 5]{duo}, L. Grafakos \cite[Chapter 4]{grafak} or A. Torchinski \cite[Chapter XI]{tor}.}
$$\tilde{f}(x) =\text{p.v.} \frac{1}{\pi}\int_{-\pi}^{\pi} \frac{f(x-y)}{2\tan \frac{y}{2}} dy =  \lim_{\varepsilon \to 0} \frac{1}{\pi}\int_{\pi>|y| >\varepsilon} \frac{f(x-y)}{2\tan \frac{y}{2}} dy.$$
This notion was extended to the non-periodic case with the definition of the {\em Hilbert transform,}
$$H f(x) =\text{p.v.} \frac{1}{\pi} \int_{-\infty}^{\infty} \frac{f(x-y)}{ y} dy =   \lim_{\varepsilon \to 0} \frac{1}{\pi} \int_{|y| > \varepsilon} \frac{f(x-y)}{ y} dy  ;$$
and then to $\mathbb{R}^d,$ with the notion of {\em Riesz transform}; see E. Stein \cite[Chap III, $\S 1$]{st1},
\begin{eqnarray}\label{Rieszdef}
\nonumber R_j f(x) &=& \text{p.v.} \; C_d \int_{\mathbb{R}^d} \frac{y_i}{|y|^{d+1}} f(x-y) dy \\
&=& \lim_{\varepsilon \to 0} C_d \int_{|y|> \varepsilon} \frac{y_j}{|y|^{d+1}} f(x-y) dy,
\end{eqnarray}
for $j= 1, \cdots, d ,$  $ f \in L^p(\mathbb{R}^d)$ with $C_d =\frac{\Gamma(\frac{d+1}{2})}{\pi^{(d+1)/2}}.$ Taking Fourier transform, we get
$$
 \widehat{(R_j f)}(\zeta) =  i \frac{\zeta_j}{|\zeta|} \hat{f}(\zeta),
$$
and thus $R_j f$ is a classical multiplier operator, with multiplier $m(y) =C_d \; i \,\frac{y_j}{|y|},$ and hence
\begin{equation}\label{ClasRieszDef}
R_j  = \frac{\partial }{\partial x_j} (-\Delta)^{-1/2},
\end{equation}
where $\Delta = \sum_{i=1}^d  \frac{\partial^2}{\partial x_i^2}$ is the Laplacian operator and $(-\Delta)^{-1/2}$ is the (classical) {\em Riesz potential} of order $1/2$.

This was later generalized to the famous {\em Calder\'on-Zygmund class} of  singular integrals:
\begin{defi} We will say that a $C^1$ function $K(x,y),$  defined off the diagonal of $\mathbb{R}^d\times \mathbb{R}^d,$ i.e. $ x \neq y$ is a Calder\'on-Zygmund kernel provided that the following conditions are satisfied:
\begin{enumerate}\label{GradC-Z}
\item [i)] $|K(x,y)|\leq \frac{C}{|x-y|^d},$
\item[ii)] $|\partial_y K(x,y)|\leq \frac{C}{|x-y|^{d+1}}.$
\end{enumerate}
associated with $K$ we define the operator $T$ by means of the formula
\begin{eqnarray*}
Tf(x) =p. v. \int_{\mathbb{R}^d} K(x,y) f(y) d =\lim\limits_{\varepsilon \to 0} \int_{|x-y| > \varepsilon} K(x,y) f(y) dy.
\end{eqnarray*}
with $f \in C_0^\infty(\mathbb{R}^d).$ We say that  $T$ is a Calder\'on-Zygmund operator if $T$ admits a continuous
extension to $L^2(\mathbb{R}^d).$
\end{defi}

For more details on this see, E. Stein \cite{st1}, \cite{duo} or \cite{grafak}.\\

The {\em Ornstein-Uhlenbeck operator} in ${\mathbb R}^d$ is a second order differential operator defined as
\begin{equation}\label{OUop}
 L = \frac{1}{2}\Delta_x - \langle x, \nabla_x \rangle= \sum_{i=1}^d \Big[\frac{1}{2} \frac{\partial^2}{\partial x_i^2} - x_i \frac{\partial }{\partial x_i}\Big],
\end{equation}
where  $\nabla_x = ( \frac{\partial}{\partial x_1},\frac{\partial}{\partial x_2},\ldots,
\frac{\partial}{\partial x_d})$ is the gradient, and $\Delta_x$ is the Laplace operator defined on the space of test functions $C_0^\infty(\mathbb{R}^d)$ of smooth functions with compact support on ${\mathbb R}^d$.

 The  Hermite polynomials in $d$-variables, $\{{\vec H}_{\nu}\}_\nu$ are eigenfunctions of $L$ with corresponding eigenvalues $\lambda_{\nu}=  - |\nu| = -  \sum_{i=1}^d \nu_i,$ i.e.
\begin{equation}\label{HermiteEigen}
 L {\vec H}_{\nu} = \lambda_{\nu} {\vec H}_{\nu} = - |\nu|  {\vec H}_{\nu}.
\end{equation}

The operator $L$ has a self-adjoint extension to $L^2(\gamma_d),$  that will be also denoted as $L$, that is,
\begin{equation}
\int_{{\mathbb R}^d}Lf(x) g(x) \gamma_d(dx) =  \int_{{\mathbb R}^d} f(x) Lg(x) \gamma_d(dx),
\end{equation}
so $L$ is the natural {\em ``symmetric'' Laplacian} in the Gaussian context.

For  $i =1, 2, \cdots, d $  let us consider the differential operators
\begin{equation}\label{naturalderv}
 \partial_{\gamma}^i  = \frac{1}{\sqrt{2}}\frac{\partial}{\partial x_i} .
 \end{equation}
$\partial_{\gamma}^i $ is not symmetric nor antisymmetric in $L^2(\gamma_d).$ In fact, its
 formal $L^2(\gamma_d )$-adjoint\footnote{In $L^2(\mathbb{R}^d),$ $\frac{\partial}{\partial x_i}$ is antisymmetric, by integration by parts.} is,
\begin{equation}\label{GaussDerAdj}
(\partial_{\gamma}^{i})^* =- \frac{1}{\sqrt{2}} e^{x_i^2}\frac{\partial}{\partial x_i} (e^{-x_i^2} I)= \sqrt{2} x_i I - \frac{1}{\sqrt{2}}\frac{\partial}{\partial x_i} ,
\end{equation}
where $I$ is the identity, which can obtained simply  by integration by parts.
Observe that $ (\partial_\gamma^i)^{*}$ can be written as
\begin{equation}\label{altgrad}
 (\partial_\gamma^i)^{*} = - e^{|x|^2} (\partial_\gamma^i e^{-|x|^2}I).
\end{equation}

Moreover, it is easy to see that
\begin{equation}\label{OUdecomp}
(-L)= \sum_{i=1}^d(\partial_\gamma^i)^{*} \partial_\gamma^i.
\end{equation}

In analogy with the classical case (\ref{ClasRieszDef}), the Gaussian Riesz transforms in ${\mathbb R}^d$ are defined spectrally, for $\,1 \leq i \leq d,$ as
\begin{equation}\label{GaussRieszdef}
{\mathcal  R}_i=  \partial_\gamma^i (-L)^{-1/2},
\end{equation}
where $(-L)^{-1/2}$ the Gaussian Riesz potential of order $1/2.$
 The meaning of this is that for any multi-index $\nu$ such that $|\nu|>0,$ its action on the Hermite polynomial $\vec{H}_\nu$ is
 \begin{equation}\label{GaussRieszAct}
\mathcal{R}_i \vec{H}_\nu = \sqrt{\frac{2}{|\nu|}}\nu_j  \vec{H}_{\nu - \vec{e}_i}
\end{equation}
where $ \vec{e_i}$ is the unitary vector with zeros in all coordinates except for the $i$-th coordinate that is one if $\nu_i >0,$ and zero otherwise.

It can be proved, for details see \cite{urbina2019},  that the kernel of $\mathcal {R}_i $ is given by
 \begin{eqnarray}
 \mathcal{K}_i(x,y) &=& \frac{1}{\pi^{d/2}\Gamma(1/2)} \int_0^1 \left( \frac{1-r^2}{-\log r}\right)^{1/2}
\frac{y_i-rx_i}{(1-r^2)^{\frac{(d+3)}{2}}} e^{- \frac{|y-rx|^2}{1-r^2}} dr,
\end{eqnarray}
and therefore, we get the integral representation of $\mathcal{R}_i ,$
\begin{eqnarray}\label{integelRepRiez}
\mathcal{R}_i f(x) &=&  \text{p.v.} \int_{{\mathbb{R}}^d} \mathcal{K}_i(x,y)
 f(y) dy \\
 \nonumber&=&  \text{p.v.}\frac{1}{\pi^{d/2}\Gamma(1/2)}\int_{{\mathbb{R}}^d}
\Big( \int_0^1  \left( \frac{1-r^2}{-\log r}\right)^{1/2}
\frac{y_i-rx_i}{(1-r^2)^{\frac{(d+3)}{2}}} e^{- \frac{|y-rx|^2}{1-r^2}} \Big)\,dr
 f(y) dy.
\end{eqnarray}

In the Gaussian case, the  {\em higher order Gaussian Riesz transforms} are defined directly,
\begin{defi} For $\beta = (\beta_1, \beta_2, \cdots, \beta_d) \in {\mathbb N}_0^d$, the  higher order Riesz transforms are defined spectrally as
\begin{equation}\label{GaussHigherRieszdef}\index{higher order Gaussian Riesz transforms}
{\mathcal R}_{\beta}= \partial^{\beta}_\gamma(-L)^{-|\beta|/2},
\end{equation}
where $|\beta| = \sum_{j=1}^d \beta_j$ and $\partial^{\beta}_\gamma = \frac{1}{2^{|\beta|/2}} \partial^{\beta_1}_{x_1} \cdots \partial^{\beta_d}_{x_d}.$
The meaning of this is that for any multi-index $\nu$ such that $|\nu|>0,$ its action on the Hermite polynomial $\vec{H}_\nu$ is
 \begin{equation}\label{GaussianHigherRieszAct}
{\mathcal  R}_\beta  \vec{H}_\nu = \Big(\frac{2}{|\nu|} \Big)^{|\beta|/2} \Big[\prod_{i=1}^d \nu_i (\nu_i -1) \cdots (\nu_i - \beta_i+1)\Big] \vec{H}_{\nu - \beta}
\end{equation}
if $\beta_i \leq \nu_i$ for all $i =1, 2, \cdots, d,$ and zero otherwise.
 \end{defi}
Observe that (\ref{GaussianHigherRieszAct}) follows directly from the definition of ${\mathcal R}_{\beta}$ since $\vec{H}_\nu$ is eigenfunction of the Ornstein-Uhlenbeck operator $-L,$ with eigenvalue $|\nu|,$ and therefore, $ (-L)^{-|\beta|/2} \vec{H}_\nu =\frac{1}{ |\nu|^{|\beta|/2}} \vec{H}_\nu.$

The higher order Gaussian Riesz transforms have kernel given by
 \begin{eqnarray*}  \label{rieszsup}
{\mathcal K}_{\beta}(x,y)
&=& \frac{1}{\pi^{d/2}\Gamma(|\beta|/2)} \int_0^1 \Big(\frac{- \log r}{1-r^2}\Big)^{\frac{|\beta| - 2}{2}}
r^{|\beta|} {\vec H}_{\beta}\Big(\frac{y-rx}{\sqrt{1-r^2}}\Big)
 \frac {e^{- \frac{|y-rx|^2}{1-r^2}}}{(1-r^2)^{d/2+1}} \frac{dr}{r},
\end{eqnarray*}
for details see \cite{urbina2019}. Therefore,
\begin{eqnarray*}
{\mathcal R}_{\beta} f(x) &=& \text{p.v.} \int_{{\mathbb{R}}^d}
{\mathcal K}_{\beta}(x,y) f(y) dy  \nonumber\\
&=&\text{p.v.} \frac{1}{\pi^{d/2}\Gamma(|\beta|/2)}\int_{{\mathbb{R}}^d}
 \int_0^1 \Big(\frac{- \log r}{1-r^2}\Big)^{\frac{|\beta| - 2}{2}}
r^{|\beta|} {\vec H}_{\beta}\Big(\frac{y-rx}{\sqrt{1-r^2}}\Big)
 \frac {e^{- \frac{|y-rx|^2}{1-r^2}}}{(1-r^2)^{d/2+1}}\, \frac{dr}{r} f(y) dy.
\end{eqnarray*}

The {\em general Gaussian singular integrals}, are  generalizations of the Gaussian higher order Riesz  transform. The first formulation of general Gaussian singular integrals was given initially by W. Urbina in \cite{ur1}. Later, S. P\'erez \cite{pe} extend it.

\begin{defi} Given a $C^1$-function $F,$ satisfying the orthogonality condition
 \begin{equation} \label{Fort1}\index{general Gaussian singular integrals}
 \int_{\mathbb{R}^d} F(x) \gamma_d(dx) = 0,
\end{equation}
and such that for every $\varepsilon >0,$ there exist constants,  $C_{\varepsilon}$ and $C_{\varepsilon}'$
such that
\begin{equation} \label{condF2}
|F(x)| \leq C_{\epsilon}e^{\epsilon|x|^2} \quad \mbox{and} \quad |\nabla F(x)| \leq C_{\epsilon}'
e^{\epsilon|x|^2}.
\end{equation}
Then, for each $m \in {\mathbb N}$ the generalized Gaussian singular integral  is defined as
\begin{equation}
T_{F,m} f(x) = \int_{{\mathbb{R}}^d} \int_0^1 \left(\frac{- \log r}{1-r^2}\right)^{\frac{m - 2}{2}}
r^{m} F\Big(\frac{y-rx}{\sqrt{1-r^2}}\Big) \frac{e^{- \frac{|y-rx|^2}{1-r^2}}}{(1-r^2)^{d/2+1}} \frac{dr}{r} f(y) dy.
\end{equation}
$T_{F,m}$ can be written as
$$ T_{F,m} f(x) = \int_{{\mathbb{R}}^d}  \mathcal{K}_{F,m} (x,y)  f(y) dy,$$
denoting,
\begin{eqnarray}
\nonumber \mathcal{K}_{F,m}  (x,y) &=&  \int_0^1 \left(\frac{- \log r}{1-r^2}\right)^{\frac{m - 2}{2}} r^{m-1} F\Big(\frac{y-rx}{\sqrt{1-r^2}}\Big)
 \frac {e^{- \frac{|y-rx|^2}{1-r^2}}}{(1-r^2)^{d/2+1}} \,dr \\
 &=& \int_0^1 \varphi_m(r) F\Big(\frac{y-rx}{\sqrt{1-r^2}}\Big)  \frac{e^{- \frac{|y-rx|^2}{1-r^2}}}{(1-r^2)^{d/2+1}} \,dr \\
\nonumber& =&\frac{1}{2} \int_0^1 \psi_m(t)
F\Big(\frac{y-\sqrt{1-t}\, x}{\sqrt{t}}\Big)
\frac {e^{-u(t)}}{t^{d/2+1}} \,dt,\\
\nonumber
\end{eqnarray}
with $\varphi_m(r)= \left(\frac{- \log r}{1-r^2}\right)^{\frac{m - 2}{2}} r^{m-1};$ and taking the change of variables $t =1-r^2,$  with $\psi_m(t) = \varphi_m(\sqrt{1-t})/\sqrt{1-t},$  and $u(t) = \frac{|\sqrt{1-t} x -y|^2}{t}.$
\end{defi}

In \cite{pe} S. P\'erez  proved that the operator
$T_{F,m}$ is a bounded operator in $L^p(\gamma_d)$, $ 1 < p < \infty$
\begin{teo}\label{strongGenGaussSingInt}
The operators $T_{F,m} $ are $L^p(\gamma_d)$ bounded for $1< p < \infty,$ that is to say there exists $C>0,$ depending only on $p$ and dimension such that
\begin{equation}
\|T_{F,m} f\|_{p,\gamma} \leq C \| f\|_{p,\gamma},
\end{equation}
for any $f \in L^p(\gamma_d).$ \\
 \end{teo}

 Now, reversing the order in (\ref{OUdecomp}), one gets another second order differential operator, that will be denoted as $\overline{L},$
\begin{equation}\label{AltOU}
(-\overline{L})= \sum_{i=1}^d\partial_\gamma^i (\partial_\gamma^i)^{*} =(-L) + d\,I= - \frac{1}{2}\Delta_x + \langle x, \nabla_x \rangle + d\,I,
\end{equation}
and therefore,
\begin{equation}\label{AltOU2}
\overline{L} = L -d\,I= \frac{1}{2}\Delta_x - \langle x, \nabla_x \rangle - d\,I.
\end{equation}
We will call $\overline{L} $ the {\em alternative
Ornstein-Uhlenbeck} operator\index{Ornstein-Uhlenbeck alternative
operator}. The Hermite polynomials $\{{\vec H}_{\nu}\}_\nu$ are
also eigenfunctions of $\overline{L},$ with eigenvalues
$\overline{\lambda}_\nu = - (|\nu|+d),$ i. e.
 \begin{equation}
\overline{L}  {\vec H}_{\nu} = (\lambda_{\nu} -d){\vec H}_{\nu} = - (|\nu|+d)  {\vec H}_{\nu}.
\end{equation}

In \cite{aimarforzaniscot}, H. Aimar, L. Forzani and R. Scotto  considered the following {\em alternative  Riesz transforms}, by taking the derivatives $ (\partial_\gamma^i)^{*} $ and Riesz potentials of the operator $(-\overline{L}),$
\begin{equation}\label{GaussRieszAltdef}
\overline{\mathcal{R}}_i= ( \partial^{i}_\gamma)^* (-\overline{L})^{-1/2}.
\end{equation}

They also considered alternative higher order Gaussian Riesz transforms, that is, for a multi-index $\beta,\; |\beta|\geq 1$  taking the representation of the gradient (\ref{altgrad}),
$$(\partial_\gamma^\beta)^*= \frac{(-1)^{|\beta|}}{2^{|\beta|/2}}e^{|x|^2}(\partial^\beta e^{-|x|^2}\,I )$$ and the Riesz potentials associated with
$\overline{L},$ these new singular integral operators are defined as follows:

\begin{defi} The alternative Gaussian Riesz transform $\overline{\mathcal R}_\beta$ for $|\beta|\geq 1$ is defined spectrally as
$$\overline{\mathcal R}_\beta f (x)= (\partial_\gamma^\beta)^* (-\overline{L})^{-|\beta|/2} f(x).$$
Thus, the  action of $\overline{\mathcal R}_\beta$ over the Hermite polynomial $ \vec{H}_\nu$ is given by
\begin{equation}\label{AltReiszAccion}
\overline{\mathcal{R}}_\beta \vec{H}_\nu = \frac{1}{2^{|\beta|/2}(|\nu|+d)^{|\beta|/2}}\vec{H}_{\nu+\beta},
\end{equation}
using  the fact that the Hermite polynomials  $\{\vec{H}_\nu\}$  are eigenfunctions of $\overline{L},$
$$(-\overline{L})^{-|\beta|/2}H_\nu= \displaystyle
\frac{1}{(|\nu|+d)^{|\beta|/2}}\vec{H}_\nu,$$
and Rodrigues' formula for the Hermite polynomials,
and therefore,
\begin{equation}\label{AltReiszAccion2}
\overline{\mathcal{R}}_\beta \vec{h}_\nu(x) = \frac{1}{(|\nu|+d)^{|\beta|/2}}\Big[\prod_{i=1}^d (\nu_i +\beta_i)(\nu_i +\beta_i-1)\cdots (\nu_i+d) \Big]^{1/2} \vec{h}_{\nu+\beta}(x).
\end{equation}
\end{defi}
It can be proved that the alternative higher order Gaussian Riesz transforms have then the following integral representation,

$$ \overline{\mathcal{R}}_{\beta}f(x)=\text{p.v.\
}e^{|x|^2}\int_{\mathbb{R}^d}\overline{\mathcal{K}}_{\beta}(x,y)f(y)\,
\gamma_d(dy) $$
where

$$ \overline{\mathcal{K}}_{\beta}(x,y)= C_{\beta}
\int_0^1\left(\frac{-\log r }{1-r^2}\right)^{\frac{|\beta|-2}{2}} r^{ d-1}
\vec{H}_{\beta}\left(\frac{x-ry}{\sqrt{1-r^2}}\right)
\frac{e^{-\frac{|x-ry|^2}{1-r^2}}}{(1-r^2)^{\frac{d}{2}+1}}\,dr,$$
for details see \cite[Chapter 9]{urbina2019}.

Now, if
\begin{eqnarray*}
K(x,y)&=&e^{|x|^{2}}\overline{\mathcal{K}}_{\beta}(x,y)e^{-|y|^{2}}\\
 &=&C_{\beta}\int_{0}^{1}\left(\frac{-\log r}{1-r^{2}}\right)^{\frac{|\beta|-2}{2}} r^{ d-1}
H_{\beta}\left(\frac{x-ry}{\sqrt{1-r^{2}}}\right)\frac{e^{\frac{-|y-rx|^{2}}{1-r^{2}}}}{(1-r^{2})^{\frac{d}{2}+1}}dr,
\end{eqnarray*}
then, $\overline{\mathcal{R}}_{\beta}$  can be written as
\begin{eqnarray}
&&\overline{\mathcal{R}}_{\beta}f(x)\\
\nonumber&& \hspace{0.5cm} = C_{\beta} \,\text{p.v.\
}\int_{\mathbb{R}^d} \int_{0}^{1}\left(\frac{-\log r}{1-r^{2}}\right)^{\frac{|\beta|-2}{2}} r^{ d-1}
H_{\beta}\left(\frac{x-ry}{\sqrt{1-r^{2}}}\right)\frac{e^{\frac{-|y-rx|^{2}}{1-r^{2}}}}{(1-r^{2})^{\frac{n}{2}+1}}dr f(y)\,dy.
\end{eqnarray}

Following the same idea to define general Gaussian singular integrals we now introduce a new class of Gaussian singular integrals, the {\em general alternative Gaussian singular integrals} as,
\begin{defi}\label{AlterGaussSingInt} Given a $C^1$-function $F,$ satisfying the orthogonality condition
 \begin{equation} \label{Fort2}
  \int_{\mathbb{R}^d} F(x) \gamma_d(dx) = 0,
\end{equation}
and such that for every $\varepsilon >0,$ there exist constants,  $C_{\varepsilon}$ and $C_{\varepsilon}'$
such that
\begin{equation} \label{condF3}
|F(x)| \leq C_{\epsilon}e^{\epsilon|x|^2} \quad \mbox{and} \quad |\nabla F(x)| \leq C_{\epsilon}'
e^{\epsilon|x|^2}.
\end{equation}
Then, for each $m \in {\mathbb N}$ the generalized alternative Gaussian singular integral  is defined as
\begin{equation}
\overline{T}_{F,m} f(x) = \int_{{\mathbb{R}}^d} \int_0^1 \left(\frac{- \log r}{1-r^2}\right)^{\frac{m - 2}{2}} r^{ d-1}
 F\Big(\frac{x-ry}{\sqrt{1-r^2}}\Big) \frac{e^{- \frac{|y-rx|^2}{1-r^2}}}{(1-r^2)^{d/2+1}} dr f(y) dy.
\end{equation}
Thus, $\overline{T}_{F,m}$ can be written as
$$\overline{T}_{F,m} f(x) = \int_{{\mathbb{R}}^d}  \overline{\mathcal{K}}_{F,m} (x,y)  f(y) dy,$$
where,
\begin{eqnarray}\label{kernelF}
\nonumber  \overline{\mathcal{K}}_{F,m} (x,y) &=&  \int_0^1 \left(\frac{- \log r}{1-r^2}\right)^{\frac{m - 2}{2}} r^{ d-1}F\Big(\frac{x-ry}{\sqrt{1-r^2}}\Big)
 \frac {e^{- \frac{|y-rx|^2}{1-r^2}}}{(1-r^2)^{d/2+1}} dr \\
 &=& \int_0^1 \varphi_m(r) F\Big(\frac{x-ry}{\sqrt{1-r^2}}\Big)  \frac{e^{- \frac{|y-rx|^2}{1-r^2}}}{(1-r^2)^{d/2+1}} dr \\
\nonumber& =& \frac{1}{2}\int_0^1 \psi_m(t)
F\Big(\frac{x-\sqrt{1-t}\, y}{\sqrt{t}}\Big)
\frac {e^{-u(t)}}{t^{d/2+1}} dt,\\
\nonumber
\end{eqnarray}
with $\varphi_m(r)= \left(\frac{- \log r}{1-r^2}\right)^{\frac{m - 2}{2}} r^{d-1};$ and after making the change of variables $t =1-r^2,$   $\psi_m(t) = \varphi_m(\sqrt{1-t})/\sqrt{1-t}= \left(\frac{- \log \sqrt{1-t}}{t}\right)^{\frac{m - 2}{2}} (\sqrt{1-t})^{d-2},$  and $u(t) = \frac{|y- \sqrt{1-t} x|^2}{t}.$
\end{defi}
Observe that the hypothesis on $F$ for the general Gaussian singular integrals, (\ref{condF2}) and the conditions on $F$ for the general alternative Gaussian singular integrals, (\ref{condF3}) are the same. We will prove the boundedness of $\overline{T}_{F,m} $ on  $L^p(\gamma_d)$, $ 1 < p < \infty$  following  \cite{pe}, for $d >1$.
\begin{teo}\label{strongGenGaussSingInt3}
The operators $ \overline{T}_{F,m} $ are $L^p(\gamma_d)$ bounded for $1< p < \infty,$  for $d>1;$ that is to say there exists $C>0,$ depending only on $p$ and dimension such that
\begin{equation}
\| \overline{T}_{F,m} f\|_{p,\gamma} \leq C \| f\|_{p,\gamma},
\end{equation}
for any $f \in L^p(\gamma_d).$ \\
 \end{teo}
 In \cite{aimarforzaniscot}, H. Aimar, L. Forzani and R. Scotto  obtained a surprising result: the alternative Riesz transforms  $\overline{\mathcal{R}}_\beta$  are weak type $(1,1)$  for all multi-index $\beta$, i. e.  independently of their orders which is a contrasting fact with respect to the anomalous behavior of the higher order Riesz transforms $\mathcal{R}_\beta$. We prove that the general alternative Gaussian singular integrals $ \overline{T}_{F,m} $ are also weak $(1,1)$ with respect to the Gaussian measure.

\begin{teo} \label{altrieszteo} For $d > 1,$ there exists a constant $ C
$ depending only on $d$ and $m$ such that for all $
\lambda >0 $ and $ f \in L^1(\gamma_d) $, we have
$$\gamma_d \Big(\Big\{ x\in \mathbb{R}^d:\; \overline{T}_{F,m}(x) > \lambda \Big\}\Big) \le \frac{C}{\lambda}
 \int_{\mathbb{R}^d} |f(y)|\gamma_d(dy).$$
\end{teo}

As usual in what follows $C$ represents a constant that is not necessarily the same in each occurrence.
\section{Proofs of the main results.}

In what follows we need the following technical results.
\begin{lemma}\label{bound1} For the function  $\psi_m(t) = \varphi_m(\sqrt{1-t})/\sqrt{1-t},$ considered  in the definition \ref{AlterGaussSingInt} we have
\begin{itemize}
\item [i)] There exists a constant $C > 0$ such that
\begin{equation}\label{ineq2}
\displaystyle|\psi_{m}(t)|\leq \frac{C}{\sqrt{1-t}},
\end{equation}
for $0\leq t<1$ and $d>1$ \item [ii)] There exists a constant $C >
0$ such that
\begin{equation}\label{ineq1}
|\psi_{m}(t)-\psi_{m}(0)|\leq \displaystyle C\frac{t}{\sqrt{1-t}},
\end{equation}
for $0\leq t<1,$ and $d>1$ where
$\psi_{m}(0)=\psi_{m}(0^{+})=2^{-(m-2)/2}$.
\end{itemize}
\end{lemma}
\begin{proof}
\begin{enumerate}

\item [ii)] It is clear, by L'Hopital's rule that
 $$\displaystyle\lim_{t\rightarrow 0^{+}}\frac{-\log \sqrt{1-t}}{t}=\lim_{t\rightarrow 0^{+}}\frac{1}{2(1-t)}=1/2,$$
 and therefore
$$\psi_{m}(0^{+})=\displaystyle\lim_{t\rightarrow 0^{+}}\psi_{m}(t)=\displaystyle\lim_{t\rightarrow 0^{+}}=\frac{\left(\frac{-\log \sqrt{1-t}}{t}\right)^{\frac{m-2}{2}}}{\sqrt{1-t}}(\sqrt{1-t})^{d-1}=2^{-(m-2)/2}.$$
Now,
\begin{eqnarray*}
|\psi_{m}(t)-\psi_{m}(0)|&=&\left|\frac{\left(\frac{-\log \sqrt{1-t}}{t}\right)^{\frac{m-2}{2}}}{\sqrt{1-t}}(\sqrt{1-t})^{d-1}-(1/2)^{(m-2)/2}\right|\\
&=&\frac{t}{\sqrt{1-t}}\left|\frac{\left(\frac{-\log \sqrt{1-t}}{t}\right)^{\frac{m-2}{2}}(\sqrt{1-t})^{d-1}-(1/2)^{(m-2)/2}\sqrt{1-t}}{t}\right|\\
&=&\frac{t}{\sqrt{1-t}}\left|B(t)\right|,
\end{eqnarray*}
where
$$B(t) = \frac{\left(\frac{-\log
\sqrt{1-t}}{t}\right)^{\frac{m-2}{2}}(\sqrt{1-t})^{d-1}-(1/2)^{(m-2)/2}\sqrt{1-t}}{t}.$$
Clearly the function  $B$ is continuous on $(0,1).$ Thus it is
enough to prove that $\displaystyle\lim_{t\rightarrow 0^{+}}B(t)$
and $\displaystyle\lim_{t\rightarrow 1^{-}}B(t)$ exist, since then
$B$ is continuous on $[0,1]$ and therefore it is bounded there.

Let us consider fist  the limit $\displaystyle\lim_{t\rightarrow 1^{-}}B(t)$. Observe that using L'Hopital's rule, can be proved that
\begin{equation}\label{ineq3}
 \displaystyle\lim_{t\rightarrow 1^{-}}(-\log \sqrt{1-t})\sqrt{1-t}=0.
\end{equation}

If $m=1$ or $m=2$, $\displaystyle\lim_{t\rightarrow 1^{-}}(-\log \sqrt{1-t})^{(m-2)/2}(\sqrt{1-t})^{d-1}=0,$ since $d>1,$ and therefore
   $$\displaystyle\lim_{t\rightarrow 1^{-}}B(t)=0.$$
On the other hand, if  $m>2$, and $m/2\leq d$, then
\begin{eqnarray*}
&&\lim_{t\rightarrow 1^{-}}(-\log \sqrt{1-t})^{(m-2)/2}(\sqrt{1-t})^{d-1}  \\
&& \hspace{3cm} = \lim_{t\rightarrow 1^{-}}\left(-\log \left(\sqrt{1-t}\right)\sqrt{1-t}\right)^{(m/2-1)}(\sqrt{1-t})^{d-m/2}=0.
\end{eqnarray*}
 Now, if  $m>2$, and $m/2>d,$ taking $n$ such that  $n\in \mathbb{N}:n\leq m/2<n+1$ then, using L'Hopital's rule $n$ times
\begin{eqnarray*}
&& \lim_{t\rightarrow 1^{-}}(-\log \sqrt{1-t})^{(m-2)/2}(\sqrt{1-t})^{d-1}\\
&& \hspace{1.5cm} =\frac{(\frac{m}{2}-1)}{(d-1)}\lim_{t\rightarrow 1^{-}}\left(-\log\sqrt{1-t}\right)^{m/2-2}(\sqrt{1-t})^{d-1}\\
&& \hspace{1.5cm} =\frac{(\frac{m}{2}-1)(\frac{m}{2}-2)}{(d-1)^{2}}\lim_{t\rightarrow
1^{-}}\left(-\log\sqrt{1-t}\right)^{m/2-3}(\sqrt{1-t})^{d-1}\\
&&\hspace{1.5cm}\vdots \\
&&\hspace{1.5cm}\cdots \\
&&\hspace{1.5cm}\vdots \\
&&\hspace{1.5cm} =\frac{(\frac{m}{2}-1)(\frac{m}{2}-2)\cdots(\frac{m}{2}-n)}{(d-1)^{n}}\lim_{t\rightarrow
1^{-}}\left(-\log\sqrt{1-t}\right)^{m/2-(n+1)}(\sqrt{1-t})^{d-1}=0,
\end{eqnarray*}
as $m/2-(n+1) < 0.$ Hence,
 $$\displaystyle\lim_{t\rightarrow 1^{-}}B(t)=0.$$
Now, we consider the limit $\displaystyle\lim_{t\rightarrow 0^{+}}B(t)$. Using again L'Hospital's rule,
\begin{eqnarray*}
&&\displaystyle\lim_{t\rightarrow 0^{+}}B(t)=\displaystyle\lim_{t\rightarrow 0^{+}}\frac{\left(\frac{-\log \sqrt{1-t}}{t}\right)^{\frac{m-2}{2}}(\sqrt{1-t})^{d-1}-(1/2)^{(m-2)/2}\sqrt{1-t}}{t}\\
&=&\displaystyle\lim_{t\rightarrow 0^{+}}(\frac{m}{2}-1)\left(\frac{-\log \sqrt{1-t}}{t}\right)^{\frac{m}{2}-2}\left(\frac{\frac{t}{2(1-t)}+\frac{1}{2}\log(1-t)}{t^{2}}\right)(\sqrt{1-t})^{d-1}\\
&&-\left(\frac{-\log
\sqrt{1-t}}{t}\right)^{\frac{m}{2}-1}(d-1)(\sqrt{1-t})^{d-2}\frac{1}{2\sqrt{1-t}}+(1/2)^{m/2-1}\frac{1}{2\sqrt{1-t}}.
\end{eqnarray*}
Observe that using L'Hopital rule twice, we have that
\begin{eqnarray*}
\lim_{t\rightarrow 0^{+}}\left(\frac{\frac{t}{2(1-t)}+\frac{1}{2}\log(1-t)}{t^{2}}\right)&=&\lim_{t\rightarrow 0^{+}}\left(\frac{\frac{1}{2(1-t)^{2}}-\frac{1}{2(1-t)}}{2t}\right)\\
&=&\frac{1}{4}\lim_{t\rightarrow 0^{+}}\frac{1}{1-t}.\lim_{t\rightarrow 0^{+}}\frac{\frac{1}{1-t}-1}{t}\\
&=&\frac{1}{4}\lim_{t\rightarrow
0^{+}}\frac{1}{1-t}.\lim_{t\rightarrow
0^{+}}\frac{1}{(1-t)^{2}}=\frac{1}{4}.
\end{eqnarray*}
Therefore,
\begin{eqnarray*}
&&\lim_{t\rightarrow 0^{+}}B(t)= \displaystyle\lim_{t\rightarrow 0^{+}}(\frac{m}{2}-1)\left(\frac{-\log \sqrt{1-t}}{t}\right)^{\frac{m}{2}-2}\left(\frac{\frac{t}{2(1-t)}+\frac{1}{2}\log(1-t)}{t^{2}}\right)(\sqrt{1-t})^{d-1}\\
&& \hspace{2cm}-\left(\frac{-\log \sqrt{1-t}}{t}\right)^{m/2-1}(d-1)(\sqrt{1-t})^{d-2}\frac{1}{2\sqrt{1-t}}+(1/2)^{m/2-1}\frac{1}{2\sqrt{1-t}}\\
&& \hspace{1.5cm} =\left(\frac{m}{2}-1\right)(1/2)^{m/2}-(d-1)(1/2)^{m/2}+(1/2)^{m/2}= \left(\frac{m}{2}+1-d\right)(1/2)^{m/2}
\end{eqnarray*}

\item[i)] It is enough to prove that $\varphi_{m}$ is bounded, i.e. there exist a constant $C>0$ such that\\
$| \varphi_{m}(r) | \leq C,$ for all $r \in [0,1]$ and $ d>1$.\\
Since $\varphi_{m}(r)$ is continuous on $(0,1)$ it is enough to see
that\\

$\displaystyle\lim_{r\rightarrow 0^{+}}  \;\varphi_{m}(r)$
and $\displaystyle\lim_{r\rightarrow 1^{-}} \;\varphi_{m}(r)$
 exist and\\
therefore $ \varphi_{m}(r)$ is a continuous function on
$[0,1]$ and then bounded.\\
  Now, from computations done in ii), we have\\

 $\displaystyle \lim_{r\rightarrow 1^{-}}  \;\varphi_{m}(r)= \displaystyle \lim_{r\rightarrow 1^{-}} \left(\frac{-\log r}{1-r^2}\right)^{\frac{m-2}{2}} r^{d-1} =2^{-(m-2)/2},$
 and\\

$\displaystyle\lim_{r\rightarrow 0^{+}}\varphi_{m}(r) = \displaystyle\lim_{r\rightarrow 0^{+}}\left(\frac{-\log r}{1-r^2}\right)^{\frac{m-2}{2}}r^{d-1}  =0.$
 \end{enumerate}
\end{proof}
In what follows we use the same notation as Proposition 4.23 \cite{urbina2019}, see also \cite{TesSon},
$$ a= a(x,y)= |x|^2 + |y|^2,\; b= b(x,y) := 2\langle x, y\rangle,$$
$$u(t) = u(t; x,y) := \frac{|y- \sqrt{1-t}x|^2}{t}=
\frac{a}{t}- \frac{\sqrt{1-t}}{t} b - |x|^2,$$
$$ t_{0}:=\frac{2\sqrt{a^{2}-b^{2}}}{a+\sqrt{a^{2}-b^{2}}} \sim \frac{\sqrt{a^2 - b^2}}{a} \sim \frac{\sqrt{a-b}}{\sqrt{a+b}} = \frac{|x-y|}{|x+y|},$$
and
$$  u_0 :=
u(t_0) = \frac{\sqrt{a^2 - b^2}}{2} + \frac{a}{2} - |x|^2 =\frac{|y|^2-|x|^2}{2}+\frac{\sqrt{a^2-b^2}}{2}.$$

For the proof of Theorem \ref{strongGenGaussSingInt3} we will need Lemma 4.36 \cite{urbina2019}, see also \cite{TesSon},
\begin{lemma}\label{CZCruCEst2}
 For every $0 \leq \eta \leq 1$ and $\nu > 0,$ there exists a constant $C$ such that if  $ \langle x, y\rangle > 0$ and $|x -y| > C_d m(x),$ we have,
\begin{equation}
\int_0^1 (u(t))^{\eta/2} e^{-\nu u(t)}
\frac{dt}{t^{3/2}\sqrt{1-t}} \leq C \frac{e^{-\nu
u_0}}{t_0^{1/2}}.
\end{equation}
\end{lemma}

Now, we are ready to prove Theorem \ref{strongGenGaussSingInt3}.

\begin{proof}
The proof follows the same scheme as the one of S. P\'erez for Theorem \ref{strongGenGaussSingInt}, see \cite{pe} (or  Theorem 9.17 of \cite{urbina2019}). As usual, we split these operators into a local and a global part,
\begin{eqnarray*}
 \overline{T}_{F,m} f(x) &=& C_d \int\limits_{|x-y|< d \,m(x)} \overline{\mathcal{K}}_{F,m} (x,y)  f(y) dy +C_d  \int\limits_{|x-y|\geq d \,m(x)} \overline{\mathcal{K}}_{F,m} (x,y)  f(y) dy  \\
&=&  \overline{T}_{F,m,L}f(x) +  \overline{T}_{F,m,G,} f(x),
\end{eqnarray*}
where
$$  \overline{T}_{F,m,L}f(x)  =   \overline{T}_{F,m}(f \chi_{B_h(\cdot)})(x) $$
 is the {\em local part}
and
$$\overline{T}_{F,m,G,} f(x)=  \overline{T}_{F,m} (f \chi_{B^c_h(\cdot)})(x) $$
 is the {\em global  part} of $ \overline{T}_{F,m},$
and
$$ B_h = B(x, C_d \,m(x)) = \{y \in \mathbb {R}^d: |y-x|< C_d \,m(x) \}$$
is an admissible ball for the Gaussian measure.

\begin{enumerate}
\item [i)] For the local part  $\overline{T}_{F,m,L},$\label{PLIntSinG} we will prove that it is always of weak type
 $(1,1).$  The needed estimates  follow from an idea that the local part differs from a Calder\'on-Zygmund singular integral by an operator that is $L^1(\gamma_d)$- bounded; in other words, the operator defined by the difference of $\overline{T}_{F,m}$ and an appropriated approximation of it (which is an operator defined as the convolution with a Calder\'on-Zygmund kernel) is $L^1(\mathbb{R}^d)$-bounded.

\begin{itemize}
\item First, observe that if $F$ satisfies the orthogonality condition  (\ref{Fort2}) and (\ref{condF3}), setting
$$ \mathcal{K}(x) = \int_0^{\infty} F \Big(- \frac{x}{t^{1/2}} \Big) e^{-|x|^2/t} \frac{dt}{t^{d/2+1}},$$
then, $\mathcal{K}$ is a Calder\'on-Zygmund kernel of convolution type (see \cite{duo}, \cite{st1} or \cite{grafak}), as the integral is absolutely convergent when $x\neq 0$. Making the change of variables $s=|x|/t^{1/2}$ we get
$$\mathcal{K}(x) := \frac{2 \int_0^{\infty} F \Big(- \frac{x}{|x|}s \Big) e^{-s^2} s^{d-1} ds}
{|x|^d} = \frac{\Omega(x)}{|x|^d},$$
with $\Omega$ homogeneous  of degree zero, and therefore $K$ is homogeneous of degree $-d.$  Moreover, $\Omega$ is $C^1$  with mean zero on $S^{d-1},$
since
\begin{eqnarray*}
 \int_ {S^{d-1}} \Omega(x') d\sigma(x') &=&  2  \int_0^{\infty} \int_ {S^{d-1}} F(-x's)   d\sigma(x') e^{-s^2} s^{d-1} ds\\
 & = & 2  \int_{\mathbb{R}^{d}} F(-y)   e^{-|y|^2} dy=0.
\end{eqnarray*}

Therefore, by the classical Calder\'on-Zygmund theory, the convolution operator defined using  convolution with the kernel $\mathcal{K},$
is continuous in  $L^p(\mathbb{R}^d), \; 1 < p < \infty$ and weak type $(1,1)$, with respect to the Lebesgue measure. Therefore, by Theorem 4.32 of \cite{urbina2019}, see also Proposition 4.3 of \cite{pe}, its local part is bounded in $L^p(\gamma_d),$ $1<p <\infty$ and of weak type $(1,1)$ with respect to $\gamma_d.$

\item Second, we need to get rid of the function $\psi_m.$ Using Lemma \ref{bound1}, we can write
\begin{eqnarray*}
\overline{\mathcal{K}}_{F,m}(x,y)
&=&\frac{1}{2}\psi_{m}(0)\int_{0}^{1}
F\left(\frac{x-\sqrt{1-t}y}{\sqrt{t}}\right)\frac{e^{-u(t)}}{t^{\frac{d}{2}+1}}dt\\
&&\hspace{-1cm}+\frac{1}{2}\int_{0}^{1}\left(\psi_{m}(t)-\psi_{m}(0)\right)
F\left(\frac{x-\sqrt{1-t}y}{\sqrt{t}}\right)\frac{e^{-u(t)}}{t^{\frac{d}{2}+1}}dt
\end{eqnarray*}
Set\begin{equation*}
\mathcal{K}_{1}(x,y):=\int_{0}^{1}F\left(\frac{x-\sqrt{1-t}y}{\sqrt{t}}\right)\frac{e^{-u(t)}}{t^{\frac{d}{2}+1}}dt.
\end{equation*}
Now, about the local part we know that $u(t) \geq |y-x|^2/t -2d,$ then, using condition (\ref{condF3}), we get
\begin{eqnarray*}
&&\left|\int_{0}^{1}\left(\psi_{m}(t)-\psi_{m}(0)\right)
F\left(\frac{x-\sqrt{1-t}y}{\sqrt{t}}\right)\frac{e^{-u(t)}}{t^{\frac{d}{2}+1}}dt\right|\\
&&\hspace{3.5cm}\leq
\int_{0}^{1}\left|\psi_{m}(t)-\psi_{m}(0)\right|
\left|F\left(\frac{x-\sqrt{1-t}y}{\sqrt{t}}\right)\right|\frac{e^{-u(t)}}{t^{\frac{d}{2}+1}}dt\\
&&\hspace{3.5cm}\leq C\int_{0}^{1}\frac{t}{\sqrt{1-t}}e^{\epsilon\frac{\left|x-\sqrt{1-t}y\right|^{2}}{t}}\frac{e^{-u(t)}}{t^{\frac{d}{2}+1}}dt\\
&&\hspace{3.5cm}=C\int_{0}^{1}\frac{e^{\epsilon
v(t)-u(t)}}{t^{\frac{d}{2}}}\frac{dt}{\sqrt{1-t}},
\end{eqnarray*}
where $v(t)=\frac{\left|x-\sqrt{1-t}y\right|^{2}}{t}$.
Observe that
$$
\epsilon v(t)-u(t)=\epsilon (v(t)-u(t))-(1-\epsilon)u(t)=\epsilon (|x|^{2}-|y|^{2})-(1-\epsilon)u(t).
$$
Then,
\begin{eqnarray*}
&&\left|\int_{0}^{1}\left(\psi_{m}(t)-\psi_{m}(0)\right)
F\left(\frac{x-\sqrt{1-t}y}{\sqrt{t}}\right)\frac{e^{-u(t)}}{t^{\frac{d}{2}+1}}dt\right|\\
&&\hspace{1.5cm}\leq Ce^{\epsilon (|x|^{2}-|y|^{2})}\int_{0}^{1}\frac{e^{-(1-\epsilon)u(t)}}{t^{\frac{d}{2}}}\frac{dt}{\sqrt{1-t}}\leq
CC_{\epsilon}\int_{0}^{1}\frac{e^{-\delta\frac{\left|x-y\right|^{2}}{t}}}{t^{\frac{d}{2}}}\frac{dt}{\sqrt{1-t}}
\end{eqnarray*}
Set
\begin{equation*}
\mathcal{K}_{2}(x):=\int_{0}^{1}\frac{e^{-\frac{\delta|x|^{2}}{t}}}{t^{\frac{d}{2}}}\frac{dt}{\sqrt{1-t}}
\end{equation*}
\item Third, we need to control the difference  between $\mathcal{K}_1$ and the Calder\'on-Zygmund kernel $\mathcal{K}.$

\underline{Claim}
\begin{equation*}
\left|\mathcal{K}_{1}(x,y)-\mathcal{K}(x-y)\right|\leq
C\frac{1+|x|^{1/2}}{|x-y|^{d-1/2}}
\end{equation*}
{\em Proof of the claim}  We need to estimate,
\begin{eqnarray*}
 && \left|\mathcal{K}_{1}(x,y)-\mathcal{K}(x-y)\right| \\
&&=\left|\int_{0}^{1}F\left(\frac{x-\sqrt{1-t}y}{\sqrt{t}}\right)\frac{e^{-u(t)}}{t^{\frac{d}{2}+1}}dt
 -\int_{0}^{\infty}F\left(\frac{y-x}{\sqrt{t}}\right)\frac{e^{-\frac{\left|x-y\right|^{2}}{t}}}{t^{\frac{d}{2}+1}}dt\right|\\
&&\leq
\left|\int_{t_{0}}^{1}F\left(\frac{x-\sqrt{1-t}y}{\sqrt{t}}\right)\frac{e^{-u(t)}}{t^{\frac{d}{2}+1}}dt
-\int_{t_{0}}^{\infty}F\left(\frac{y-x}{\sqrt{t}}\right)\frac{e^{-\frac{\left|x-y\right|^{2}}{t}}}{t^{\frac{d}{2}+1}}dt\right|\\
&&\hspace{1cm}+\left|\int_{0}^{t_{0}}F\left(\frac{x-\sqrt{1-t}y}{\sqrt{t}}\right)\frac{e^{-u(t)}}{t^{\frac{d}{2}+1}}dt
-\int_{0}^{t_{0}}F\left(\frac{y-x}{\sqrt{t}}\right)\frac{e^{-\frac{\left|x-y\right|^{2}}{t}}}{t^{\frac{d}{2}+1}}dt\right|\\
&& =(\mathrm{I})+(\mathrm{II}).
\end{eqnarray*}
Using again the notation of Proposition 4.23 of \cite{urbina2019}, and  the fact that on the local part $u(t) \geq |y-x|^2/t -2d,$ there is a $\delta >0$ such that,
\begin{eqnarray*}
(\mathrm{I})&\leq
&\int_{t_{0}}^{1}\left|F\left(\frac{x-\sqrt{1-t}y}{\sqrt{t}}\right)\right|\frac{e^{-u(t)}}{t^{\frac{d}{2}+1}}dt
+\int_{t_{0}}^{\infty}\left|F\left(\frac{y-x}{\sqrt{t}}\right)\right|\frac{e^{-\frac{\left|x-y\right|^{2}}{t}}}{t^{\frac{d}{2}+1}}dt\\
&\leq &C_{1}\int_{t_{0}}^{1}\frac{e^{\epsilon
v(t)-u(t)}}{t^{\frac{(d-1)}{2}}}\frac{dt}{t^{\frac{3}{2}}}
+C_{2}\int_{t_{0}}^{\infty}e^{(\epsilon-1)\frac{\left|x-y\right|^{2}}{t}}\frac{dt}{t^{\frac{d}{2}+1}}\\
&\leq
&C_{1}e^{\epsilon(|x|^{2}-|y|^{2})}\int_{t_{0}}^{1}\frac{e^{-(1-\epsilon)u(t)}}{t^{\frac{(d-1)}{2}}}\frac{dt}{t^{\frac{3}{2}}}
+C_{2}\int_{t_{0}}^{\infty}e^{(\epsilon-1)\frac{\left|x-y\right|^{2}}{t}}\frac{dt}{t^{\frac{d}{2}+1}}\\
&\leq
&C_{\epsilon}\int_{t_{0}}^{1}\frac{e^{-(1-\epsilon)\frac{\left|x-y\right|^{2}}{t}}}{t^{\frac{(d-1)}{2}}}\frac{dt}{t^{\frac{3}{2}}}
+C_{2}\int_{t_{0}}^{\infty}e^{(\epsilon-1)\frac{\left|x-y\right|^{2}}{t}}\frac{dt}{t^{\frac{d}{2}+1}}\\
&\leq &2C\int_{t_{0}}^{\infty}\frac{e^{-\delta\frac{\left|x-y\right|^{2}}{t}}}{t^{\frac{(d-1)}{2}}}\frac{dt}{t^{\frac{3}{2}}}\leq C\frac{1}{|x-y|^{d-1}}.\frac{1}{t_{0}^{1/2}}\leq
C\frac{1+|x|^{1/2}}{|x-y|^{d-\frac{1}{2}}}.
\end{eqnarray*}
Now, we need to bound $(\mathrm{II})$.

Set
$w(s)=x-\sqrt{1-s}y,\hspace{0.2cm}$$z(s)=y-\sqrt{1-s}x,\hspace{0.2cm}$
then $\hspace{0.2cm}w'(s)=\frac{y}{2\sqrt{1-s}}\hspace{0.2cm}$ and
$\hspace{0.2cm}z'(s)=\frac{x}{2\sqrt{1-s}}$.\\
Then,
\begin{eqnarray*}
&&\left|F\left(\frac{x-\sqrt{1-t}y}{\sqrt{t}}\right)e^{-u(t)}-F\left(\frac{x-y}{\sqrt{t}}\right)e^{-\frac{\left|x-y\right|^{2}}{t}}\right|=\left|F\left(\frac{w(t)}{\sqrt{t}}\right)e^{\frac{-|z(t)|^{2}}{t}}-F\left(\frac{w(0)}{\sqrt{t}}\right)e^{-\frac{\left|z(0)\right|^{2}}{t}}\right|\\
&& \quad =\left|\int_{0}^{t}\frac{\partial}{\partial
s}\left(F\left(\frac{w(s)}{\sqrt{t}}\right)e^{-\frac{|z(s)|^{2}}{t}}\right)ds\right|\\
&& \quad =\left|\int_{0}^{t}\left\langle\frac{w'(s)}{\sqrt{t}},\nabla
F\left(\frac{w(s)}{\sqrt{t}}\right) \right\rangle
e^{-\frac{|z(s)|^{2}}{t}}-2\left\langle
z'(s),\frac{z(s)}{t}\right\rangle
F\left(\frac{w(s)}{\sqrt{t}}\right)e^{-\frac{|z(s)|^{2}}{t}}ds\right|\\
&& \quad \leq \int_{0}^{t}\left|\frac{w'(s)}{\sqrt{t}}\right|\left|\nabla
F\left(\frac{w(s)}{\sqrt{t}}\right)\right|
e^{-\frac{|z(s)|^{2}}{t}}ds+2\int_{0}^{t}\left|\frac{z'(s)}{\sqrt{t}}\right|\left|\frac{z(s)}{\sqrt{t}}\right|
\left|F\left(\frac{w(s)}{\sqrt{t}}\right)\right|e^{-\frac{|z(s)|^{2}}{t}}ds\\
&& \quad \leq \int_{0}^{t}\frac{|y|}{\sqrt{t}2\sqrt{1-s}}e^{\epsilon'\frac{|w(s)|^{2}}{t}-\frac{|z(s)|^{2}}{t}}ds+2\int_{0}^{t}\frac{|x|}{\sqrt{t}2\sqrt{1-s}}
\left|\frac{z(s)}{\sqrt{t}}\right|e^{\epsilon'\frac{|w(s)|^{2}}{t}-\frac{|z(s)|^{2}}{t}}ds\\
&& \quad \leq \frac{|y|}{2\sqrt{t}}\int_{0}^{t}\frac{1}{\sqrt{1-s}}e^{\epsilon'\frac{s}{t}(|x|^{2}-|y|^{2})-(1-\epsilon')\frac{|z(s)|^{2}}{t}}ds+
\frac{|x|}{\sqrt{t}}\int_{0}^{t}\frac{1}{\sqrt{1-s}}
\left|\frac{z(s)}{\sqrt{t}}\right|e^{\epsilon'\frac{s}{t}(|x|^{2}-|y|^{2})-(1-\epsilon')\frac{|z(s)|^{2}}{t}}ds\\
&& \quad \leq
\frac{C|y|}{2\sqrt{t}}\int_{0}^{t}\frac{1}{\sqrt{1-s}}e^{-(1-\epsilon')\frac{|z(s)|^{2}}{t}}ds+
\frac{C|x|}{\sqrt{t}}\int_{0}^{t}\frac{1}{\sqrt{1-s}}
\left|\frac{z(s)}{\sqrt{t}}\right|e^{-\frac{(1-\epsilon')}{2}\frac{|z(s)|^{2}}{t}}e^{-\frac{(1-\epsilon')}{2}\frac{|z(s)|^{2}}{t}}ds\\
&& \quad \leq
\frac{C|y|}{2\sqrt{t}}\int_{0}^{t}\frac{1}{\sqrt{1-s}}e^{-(1-\epsilon')\frac{|z(s)|^{2}}{t}}ds+
\frac{C|x|}{\sqrt{t}}\int_{0}^{t}\frac{1}{\sqrt{1-s}}e^{-\frac{(1-\epsilon')}{2}\frac{|z(s)|^{2}}{t}}ds.\\
\end{eqnarray*}
On the other hand, in the local part, we have
\begin{eqnarray*}
\frac{|z(s)|^{2}}{t}&=&\frac{|y-\sqrt{1-s}x|^{2}}{t}=\frac{|(y-x)-(\sqrt{1-s}-1)x|^{2}}{t}\\
&\geq &\frac{\left(|(y-x)|-|(\sqrt{1-s}-1)x|\right)^{2}}{t}\\
&\geq &\frac{|(y-x)|^{2}}{t}-2\frac{|y-x||x|(\sqrt{1-s}-1)}{t}\geq
\frac{|(y-x)|^{2}}{t}-2|y-x||x|\geq
\frac{|(y-x)|^{2}}{t}-2C.\\
\end{eqnarray*}
Thus,
\begin{eqnarray*}
&&\left|F\left(\frac{x-\sqrt{1-t}y}{\sqrt{t}}\right)e^{-u(t)}-F\left(\frac{x-y}{\sqrt{t}}\right)e^{-\frac{\left|x-y\right|^{2}}{t}}\right|\\
&& \quad \leq \frac{C|y|}{2\sqrt{t}}\int_{0}^{t}\frac{1}{\sqrt{1-s}}e^{-(1-\epsilon')\frac{|y-x|^{2}}{t}}ds+
\frac{C|x|}{\sqrt{t}}\int_{0}^{t}\frac{1}{\sqrt{1-s}}e^{-\frac{(1-\epsilon')}{2}\frac{|y-x|^{2}}{t}}ds\\
&& \quad \leq \frac{C(|y|+|x|)}{\sqrt{t}}\int_{0}^{t}\frac{1}{\sqrt{1-s}}e^{-\frac{(1-\epsilon')}{2}\frac{|y-x|^{2}}{t}}ds\\
&& \quad =\frac{C(|y|+|x|)}{\sqrt{t}}e^{-\delta\frac{|y-x|^{2}}{t}}\int_{0}^{t}\frac{1}{\sqrt{1-s}}ds=\frac{C(|y|+|x|)}{\sqrt{t}}e^{-\delta\frac{|y-x|^{2}}{t}}2(1-\sqrt{1-t}),
\end{eqnarray*}
where $ \delta=\frac{(1-\epsilon')}{2}>0$.\\
 Hence,
\begin{eqnarray*}
(\mathrm{II}) &=&\int_{0}^{t_{0}}\left|F\left(\frac{x-\sqrt{1-t}y}{\sqrt{t}}\right)e^{-u(t)}-F\left(\frac{x-y}{\sqrt{t}}\right)e^{-\frac{\left|x-y\right|^{2}}{t}}\right|\frac{dt}{t^{d/2+1}}\\
&\leq
&2C(|y|+|x|)\int_{0}^{t_{0}}\frac{e^{-\delta\frac{|x-y|^{2}}{t}}}{t^{d/2}}\frac{(1-\sqrt{1-t})}{t}\frac{dt}{\sqrt{t}}\\
&\leq
&2C(|y|+|x|)\int_{0}^{t_{0}}\frac{e^{-\delta\frac{|x-y|^{2}}{t}}}{t^{d/2}}\frac{dt}{\sqrt{t}}\leq
\frac{2C(|y|+|x|)}{|x-y|^{d}}\int_{0}^{t_{0}}\frac{dt}{\sqrt{t}}=\frac{2C(|y|+|x|)}{|x-y|^{d}}t_{0}^{1/2}.\\
\end{eqnarray*}
Now, using that
 \begin{equation*}
 t_{0}=\frac{2\sqrt{a^{2}-b^{2}}}{a+\sqrt{a^{2}-b^{2}}}\leq
 \frac{2\sqrt{a^{2}-b^{2}}}{a}=\frac{2|x+y||x-y|}{|x|^{2}+|y|^{2}}
\end{equation*}
we get,
\begin{eqnarray*}
\frac{C(|y|+|x|)}{|x-y|^{d}}t_{0}^{1/2}&\leq
&\frac{C(|y|+|x|)}{|x-y|^{d}}\left(\frac{2|x+y||x-y|}{|x|^{2}+|y|^{2}}\right)^{1/2}=\frac{C(|y|+|x|)}{|x-y|^{d-\frac{1}{2}}}\left(\frac{|x+y|^{1/2}}{\left(|x|^{2}+|y|^{2}\right)^{1/2}}\right)\\
&=&\frac{C|y|}{|x-y|^{d-\frac{1}{2}}}\left(\frac{|x+y|^{1/2}}{\left(|x|^{2}+|y|^{2}\right)^{1/2}}\right)+\frac{C|x|}{|x-y|^{d-\frac{1}{2}}}\left(\frac{|x+y|^{1/2}}{\left(|x|^{2}+|y|^{2}\right)^{1/2}}\right)\\
&\leq &\frac{C|y|}{|x-y|^{d-\frac{1}{2}}}\left(\frac{|x+y|^{1/2}}{\left(|y|^{2}\right)^{1/2}}\right)+\frac{C|x|}{|x-y|^{d-\frac{1}{2}}}\left(\frac{|x+y|^{1/2}}{\left(|x|^{2}\right)^{1/2}}\right)\\
&=&\frac{2C}{|x-y|^{d-\frac{1}{2}}}\left(|x+y|^{1/2}\right)\leq \frac{2C}{|x-y|^{d-\frac{1}{2}}}\left(|x|^{1/2}+|y|^{1/2}\right).\\
\end{eqnarray*}
Also in the local part, we have
 \begin{equation*}
\left||x|-|y|\right|\leq |x-y|\leq C_{1},\; \mbox{i.e.} \;
|y|\leq C_{1} +|x|.
\end{equation*}
Then,
\begin{eqnarray*}
\frac{C(|y|+|x|)}{|x-y|^{d}}t_{0}^{1/2}&\leq &\frac{C}{|x-y|^{d-\frac{1}{2}}}\left(|x|^{1/2}+(C_{1}^{1/2}+|x|^{1/2})\right)\\
&\leq
&\frac{C|x|^{1/2}+C}{|x-y|^{d-\frac{1}{2}}}=\frac{C\left(1+|x|^{1/2}\right)}{|x-y|^{d-\frac{1}{2}}},
\end{eqnarray*}
and therefore,
\begin{equation*}
(\mathrm{II})\leq
\frac{C\left(1+|x|^{1/2}\right)}{|x-y|^{d-\frac{1}{2}}}.
\end{equation*}
Set
 \begin{equation*}
\mathcal{K}_{3}(x,y):=\frac{1+|x|^{\frac{1}{2}}}{|x-y|^{d-\frac{1}{2}}}.
\end{equation*}
Observe that $K_3(x,y)$ defines a function in the variable $x$ which is $L^1(\mathbb{R}^d),$ uniformly in the variable $y$.\\

\end{itemize}
Hence, writing $\overline{\mathcal{K}}_{F,m} (x,y)$ as
\begin{eqnarray*}
\overline{\mathcal{K}}_{F,m} (x,y) &=&  \frac{1}{2}\int_0^1 \psi_m(t) F\Big(\frac{y-\sqrt{1-t}\, x}{\sqrt{t}}\Big) \frac {e^{-u(t)}}{t^{d/2+1}} dt,\\
&=& \frac{1}{2}\psi_m(0) \int_0^1 F\Big(\frac{y-\sqrt{1-t}\, x}{\sqrt{t}}\Big) \frac {e^{-u(t)}}{t^{d/2+1}} dt \\
&& \hspace{2.5cm} + \frac{1}{2} \int_0^1 (\psi_m(t) - \psi_m(0)) F\Big(\frac{y-\sqrt{1-t}\, x}{\sqrt{t}}\Big) \frac {e^{-u(t)}}{t^{d/2+1}} dt\\
&=& \frac{1}{2}\psi_m(0)\Big[ \int_0^1 F\Big(\frac{y-\sqrt{1-t}\, x}{\sqrt{t}}\Big) \frac {e^{-u(t)}}{t^{d/2+1}}  dt - \int_0^\infty F\Big(\frac{y-x}{\sqrt{t}}\Big) \frac {e^{-\frac{|x-y|^2}t}}{t^{d/2+1}} dt \Big] \\
&&+ \frac{1}{2} \psi_m(0) \int_0^\infty  F\Big(\frac{y-x}{\sqrt{t}}\Big) \frac {e^{-\frac{|x-y|^2}t}}{t^{d/2+1}} dt \\
&& \hspace{3.5cm} +   \frac{1}{2} \int_0^1 (\psi_m(t) - \psi_m(0))
F\Big(\frac{y-\sqrt{1-t}\, x}{\sqrt{t}}\Big) \frac
{e^{-u(t)}}{t^{d/2+1}} dt.
\end{eqnarray*}
Using the estimates above, we conclude that  the local part $T_{F,m,L}$ can be bounded as
\begin{eqnarray*}
|\overline{T}_{F,m,L} f(x)| &=&  |\overline{T}_{F,m} f(\chi_{B_h(x)})(x)|  =  \Big| \int_{B_h(x)}  \overline{{\mathcal K}}_{F,m} (x,y)  f(y) \;dy \Big|\\
&\leq& C \int_{B_h(x)}  {\mathcal K}_3(x,y)  |f(y)| \;dy  + C \Big|  p.v.\; \int_{B_h(x)}  {\mathcal K}(x-y)   f(y) \;dy \Big| \\
&& \hspace{4cm} + C\int_{B_h(x)}  {\mathcal K}_2(x-y) | f(y) |\;dy \\
&=& (I) + (II) + (III).\\
\end{eqnarray*}

By Theorem 4.32 of \cite{urbina2019}, $(II)$ is bounded in $L^p(\gamma_d),$ $1<p <\infty$ and of weak type $(1,1)$ with respect to $\gamma_d.$ So it remains to prove that $(I)$ and $(III)$ are also bounded. In order to do that, we need to use a covering lemma, Lemma 4.3 of \cite{urbina2019}; taking a countable family of admissible balls $\mathcal{F}.$

Given $B \in  \mathcal{F},$ if  $x \in B$ then $B_h(x) \subset \hat{B}$, and therefore,
\begin{eqnarray*}
(I) &=&  (1+|x|^{1/2}) \sum_{k=0}^\infty \int_{2^{-(k+1)}C_d m(x) < |x-y| < 2^{-k}C_d m(x)} \frac{|f(y)| \chi_ {\hat{B}}}{|x-y|^{d-1/2}} dy\\
&\leq& C_d 2^d \mathcal{M}(f\chi_{\hat{B}})(x) (1+|x|^2) m(x)^{1/2}  \sum_{k=0}^\infty 2^{-(k+1)/2} \leq C  \mathcal{M}(f\chi_{\hat{B}})(x) (\chi_{B_h(\cdot)})(x),
\end{eqnarray*}
where $\mathcal{M}(g)$ is the classical Hardy-Littlewood maximal function of the function $g.$

On the other hand, let us consider $\varphi(y) =C_\delta
e^{-\delta|y|^2},$ where $C_\delta$ is a constant such that
$\int_{\mathbb{R}^d} \varphi(y) dy =1$. $\varphi$ is a
non-increasing radial function, and given $t>0,$ we rescale this
function as $\varphi_{\sqrt{t}} (y) = t^{-d/2} \phi(y/\sqrt{t}),$
and, since $0\leq \varphi \in L^1(\mathbb{R}^d),$
$\{\varphi_{\sqrt{t}} \}_{t>0}$ is a classical  approximation of
the identity in $\mathbb{R}^d.$ Then, since $\int_0^1
(1/\sqrt{1-t}) dt <\infty,$
\begin{eqnarray*}
(III) &=& \int_{B_h(x)}  \mathcal{K}_2(x-y) | f(y) |\;dy  =  \int_{B_h(x)}  \Big(\int_0^1 \varphi_{\sqrt{t}}(x-y) \frac{dt}{\sqrt{1-t}}\Big) |f(y)| dy\\
&\leq&  \int_{B_h(x)} \Big(\sup_{t>0}  \varphi_{\sqrt{t}}(x-y) \Big) \Big(\int_0^1\frac{dt}{\sqrt{1-t}}\Big) |f(y)| dy\\
&\leq& C \int_{B_h(x)} \Big(\sup_{t>0}  \varphi_{\sqrt{t}}(x-y) \Big) |f(y)| dy.
\end{eqnarray*}

Again, using the family  $\mathcal{F}$ if  $x \in B$ then $B_h(x) \subset \hat{B}$, then, by a similar argument as before,
\begin{eqnarray*}
(III) &=& \int_{B_h(x)}  \mathcal{K}_2(x-y) | f(y) |\;dy  \leq C \int_{\mathbb{R}^d} \Big(\sup_{t>0}  \varphi_{\sqrt{t}}(x-y) \Big)  |f(y)| \chi_{\hat{B}}(y) dy
\end{eqnarray*}
which yields, using Theorem 4 in Stein's book \cite[Chapter II \S 4.]{st1}, we get
\begin{eqnarray*}
(III) &=& \int_{B_h(x)}  \mathcal{K}_2(x-y) | f(y) |\;dy  \leq  \sum_{B \in \mathcal{F}}  \sup_{t>0}  \Big| (\varphi_{\sqrt{t}} *  |f \chi_{\hat{B}}|)(x) \Big| \chi_{B}(x)\\
&\leq & \sum_{B \in \mathcal{F}} \mathcal{M}(f \chi_{\hat{B}})(x) \chi_{B}(x).
\end{eqnarray*}

Therefore,  the local part $T_{F,m,L}$ is bounded in $L^p(\gamma_d),$ $1<p <\infty$ and of weak type $(1,1)$ with respect to $\gamma_d.$\\

\item[ii)] Now, for the global part $\overline{T}_{F,m, G},$\label{PGIntSinG} we will prove that it  is $L^p(\gamma_d)$-bounded for all  $1 < p < \infty.$ The idea will be to exploit the size of the kernel and treat $\overline{T}_{F,m, G}$ as a positive operator.

Observe that, from Lemma \ref{bound1},
$$\displaystyle|\psi_{m}(t)|\leq \frac{C}{\sqrt{1-t}}.$$
Hence, using (\ref{condF3}) and $v(t)=\frac{|x-\sqrt{1-t}y|^{2}}{t},$ we get
\begin{eqnarray*}
 \overline{{\mathcal K}}_{F,m} (x,y) | &\leq& C \int_0^1 \Big| F\Big(\frac{x-\sqrt{1-t}\, y}{\sqrt{t}}\Big)  \Big| \frac {e^{-u(t)}}{t^{d/2+1}} \frac{dt}{\sqrt{1-t}}\\
 & \leq&  C_\epsilon \int_0^1 e^{\epsilon v(t)}  \frac {e^{-u(t)}}{t^{d/2+1}} \frac{dt}{\sqrt{1-t}},
\end{eqnarray*}
for some $\epsilon>0$ to be determined.\\

 Let us take $E_{x}=\{y:\langle x,y\rangle> 0\}$ and consider two cases:

\begin{itemize}
\item Case $\#$1:  $b = 2\langle x, y\rangle \leq 0.$
Now, as $v(t)=\frac{a}{t}-\frac{\sqrt{1-t}b}{t}-|y|^{2},$
$$\frac{a}{t}- |y|^2 \leq v(t)= \frac{a}{t} - \frac{\sqrt{1-t}}{t} b -|y|^2 \leq \frac{2a}{t},$$
and so, the change of variables $s= a(\frac{1}{t}-1)$ gives
\begin{eqnarray*}
 \int_{0}^{1}\frac{e^{\epsilon
 v(t)-u(t)}}{t^{\frac{d}{2}+1}}\frac{dt}{\sqrt{1-t}}&=&e^{|x|^{2}-|y|^{2}}
 \int_{0}^{1}\frac{e^{-(1-\epsilon)
 v(t)}}{t^{\frac{ d}{2}+1}}\frac{dt}{\sqrt{1-t}}\\
 &\leq &e^{|x|^{2}-|y|^{2}}
 \int_{0}^{1}\frac{e^{-(1-\epsilon)(\frac{a}{t}-|y|^{2})
 }}{t^{\frac{d}{2}+1}}\frac{dt}{\sqrt{1-t}}\\
&= & e^{|x|^{2}-|y|^{2}+|y|^{2}-\epsilon|y|^{2}}
\int_{0}^{1}\frac{e^{-(1-\epsilon)(\frac{a}{t})
 }}{t^{\frac{d}{2}+1}}\frac{dt}{\sqrt{1-t}}\\
  &= &\frac{e^{|x|^{2}-|y|^{2}+|y|^{2}-\epsilon|y|^{2}}}{a^{\frac{d}{2}}}
 \int_{0}^{\infty}e^{-(1-\epsilon)
 (s+a)}(s+a)^{\frac{d}{2}-\frac{1}{2}}\frac{ds}{s^{\frac{1}{2}}}\\
 &= &\frac{e^{-|y|^{2}+\epsilon|x|^{2}}}{a^{\frac{d}{2}}}
 \int_{0}^{\infty}e^{-(1-\epsilon)s
 }(s+a)^{\frac{d}{2}-\frac{1}{2}}\frac{ds}{s^{\frac{1}{2}}}\\
&\leq &\frac{e^{-|y|^{2}+\epsilon|x|^{2}}}{a^{\frac{d}{2}}} C
\left(\int_{0}^{\infty}e^{-(1-\epsilon)s
 }s^{\frac{d}{2}-1}ds+a^{\frac{d}{2}-\frac{1}{2}}\int_{0}^{\infty}e^{-(1-\epsilon)s
 }s^{\frac{1}{2}-1}ds\right)\\
 &=&C_{\epsilon} \frac{e^{-|y|^{2}+\epsilon|x|^{2}}}{a^{\frac{d}{2}}}
\left(\Gamma(\frac{d}{2})+a^{\frac{d}{2}-\frac{1}{2}}\Gamma(\frac{1}{2})\right)=C_{\epsilon}e^{-|y|^{2}+\epsilon|x|^{2}}
\left(\frac{1}{a^{\frac{d}{2}}}+\frac{1}{a^{\frac{1}{2}}}\right)\\
&\leq &C_{\epsilon}e^{-|y|^{2}+\epsilon|x|^{2}},
 \end{eqnarray*}
 as $a>\frac{1}{2}$ over the global region. Thus,
\begin{eqnarray*}
\left|\overline{\mathcal{K}}_{F,m}(x,y)\right|&\leq
&\int_{0}^{1}\left|F\left(\frac{x-\sqrt{1-t}y}{\sqrt{t}}\right)\right|\frac{e^{-u(t)}}{t^{\frac{d}{2}+1}}\frac{dt}{\sqrt{1-t}}\\
&\leq
&C_{\epsilon}\int_{0}^{1}\frac{e^{\epsilon v(t)-u(t)}}{t^{\frac{d}{2}+1}}\frac{dt}{\sqrt{1-t}}\leq  C_{\epsilon}e^{-|y|^{2}+\epsilon|x|^{2}}.
\end{eqnarray*}
Therefore,
\begin{eqnarray*}
\int_{\mathbb{R}^{d}}\left(\int_{B_{h}^{c}(x)\cap
E_{x}^{c}}\left|\overline{\mathcal{K}}_{F,m}(x,y)\right|\left|f(y)\right|dy\right)^{p}e^{-|x|^{2}}dx&\leq
&C_{\epsilon}\int_{\mathbb{R}^{d}}\left(\int_{B_{h}^{c}(x)\cap
E_{x}^{c}}e^{-|y|^{2}+\epsilon|x|^{2}}\left|f(y)\right|dy\right)^{p}e^{-|x|^{2}}dx\\
&\leq
&C_{\epsilon}\int_{\mathbb{R}^{d}}\left(\int_{B_{h}^{c}(x)}\left|f(y)\right|e^{-|y|^{2}}dy\right)^{p}e^{\epsilon|x|^{2}p-|x|^{2}}dx\\
&\leq
&C_{\epsilon}\int_{\mathbb{R}^{d}}\left(\int_{\mathbb{R}^{d}}\left|f(y)\right|^{p}\gamma(dy)\right)e^{(\epsilon p-1)|x|^{2}}dx\\
&=&C_{\epsilon}\left\|f\right\|^{p}_{p,\gamma}\int_{\mathbb{R}^{d}}e^{(\epsilon
p-1)|x|^{2}}dx =C_{\epsilon}\left\|f\right\|^{p}_{p,\gamma},
\end{eqnarray*}
for $\epsilon<1/p.$\\

\item Case $\#$2:  $b = 2\langle x, y\rangle > 0$. Consider again
$$u_0 =u(t_0)=\frac{|y|^2-|x|^2}{2}+\frac{\sqrt{a^2-b^2}}{2} \leq (a^2-b^2)^{1/2}.$$
Let $y \in B_{h}^{c}(x),$
\begin{eqnarray*}
|\overline{K}_{F,m}(x,y)|&\leq&C\int_{0}^{1}\left|F\left(\frac{x-\sqrt{1-t}y}{\sqrt{t}}\right)\right|\frac{e^{-u(t)}}{t^{\frac{d}{2}+1}}\frac{dt}{\sqrt{1-t}}\\
&\leq& C\int_{0}^{1}\frac{e^{\frac{\epsilon|x-\sqrt{1-t}y|^{2}}{t}-u(t)}}{t^{\frac{d}{2}+1}}\frac{dt}{\sqrt{1-t}}=C\int_{0}^{1}\frac{e^{\epsilon v(t)-u(t)}}{t^{\frac{d}{2}+1}}\frac{dt}{\sqrt{1-t}}\\
&=&Ce^{\epsilon(|x|^{2}-|y|^{2})}\int_{0}^{1}\frac{e^{-(1-\epsilon)u(t)}}{t^{\frac{d-1}{2}}}\frac{dt}{t^{\frac{3}{2}}\sqrt{1-t}}\\
&=&Ce^{\epsilon(|x|^{2}-|y|^{2})}\int_{0}^{1}\frac{e^{-(\frac{d-1}{d})u(t)}}{t^{\frac{d-1}{2}}}\frac{e^{\epsilon u(t)-\frac{u(t)}{d}}dt}{t^{\frac{3}{2}}\sqrt{1-t}}
\end{eqnarray*}
Now, we know that
\begin{eqnarray*}
\frac{e^{-(\frac{d-1}{d})u(t)}}{t^{\frac{d-1}{2}}}&=&\left(\frac{e^{-u(t)}}{t^{\frac{d}{2}}}\right)^{\frac{d-1}{d}}\leq\left(\frac{e^{-u(t_{0})}}{t_{0}^{\frac{d}{2}}}\right)^{\frac{d-1}{d}}=\frac{e^{-(\frac{d-1}{d})u(t_{0})}}{t_{0}^{\frac{d-1}{2}}}
\end{eqnarray*}
Then, by  Lemma \ref{CZCruCEst2}, taking $v=\frac{1}{d}-\epsilon>0$
\begin{eqnarray*}
|\overline{K}_{F,m}(x,y)|&\leq&Ce^{\epsilon(|x|^{2}-|y|^{2})}\frac{e^{-(\frac{d-1}{d})u(t_{0})}}{t_{0}^{\frac{d-1}{2}}}\int_{0}^{1}\frac{e^{(\epsilon -\frac{1}{d})u(t)}dt}{t^{\frac{3}{2}}\sqrt{1-t}}\\
&\leq&Ce^{\epsilon(|x|^{2}-|y|^{2})}\frac{e^{-(\frac{d-1}{d})u(t_{0})}}{t_{0}^{\frac{d-1}{2}}}\frac{e^{(\epsilon -\frac{1}{d})u(t_{0})}}{t_{0}^{\frac{1}{2}}}=Ce^{\epsilon(|x|^{2}-|y|^{2})}\frac{e^{-(1-\epsilon)u(t_{0})}}{t_{0}^{\frac{d}{2}}}.
\end{eqnarray*}
Thus,
\begin{eqnarray*}
&&\int_{\mathbb{R}^{d}}\left(\int_{B_{h}^{c}(x)\cap E_{x}}|\overline{K}_{F,m}(x,y)f(y)|dy\right)^{p}e^{-|x|^{2}}dx\\
&&\hspace{2cm} \leq C\int_{\mathbb{R}^{d}}\left(\int_{B_{h}^{c}(x)\cap E_{x}}e^{\epsilon(|x|^{2}-|y|^{2})}\frac{e^{-(1-\epsilon)u(t_{0})}}{t_{0}^{\frac{d}{2}}}|f(y)|dy\right)^{p}e^{-|x|^{2}}dx\\
&&\hspace{2cm} =C\int_{\mathbb{R}^{d}}\left(\int_{B_{h}^{c}(x)\cap E_{x}}e^{(\epsilon-\frac{1}{p})(|x|^{2}-|y|^{2})}\frac{e^{-(1-\epsilon)u(t_{0})}}{t_{0}^{\frac{d}{2}}}|f(y)|e^{-\frac{|y|^{2}}{p}}dy\right)^{p}dx.
\end{eqnarray*}
Therefore, it is enough to check that the operator defined using the kernel,
 $$\widetilde{K}(x,y)=\displaystyle e^{(\epsilon-\frac{1}{p})(|x|^{2}-|y|^{2})}\frac{e^{-(1-\epsilon)u(t_{0})}}{t_{0}^{\frac{d}{2}}}\chi_{B_{h}^{c}(x)}(y),$$
is of strong type p with respect to the Lebesgue measure. Using the inequality  $||y|^2-|x|^2| \leq |x+y||x-y|,$ definition of $t_{0}$
and that, as $b>0,$ then on the global region, $|x+y||x-y| \geq d,$ we conclude that
\begin{eqnarray*}
\displaystyle e^{(\epsilon-\frac{1}{p})(|x|^{2}-|y|^{2})}\frac{e^{-(1-\epsilon)u(t_{0})}}{t_{0}^{\frac{d}{2}}}&=&\frac{1}{t_{0}^{\frac{d}{2}}}e^{\left[(\frac{1}{p}-\epsilon)-\frac{(1-\epsilon)}{2}\right](|y|^{2}-|x|^{2})}e^{-\frac{(1-\epsilon)}{2}|x+y||x-y|}\\
&\leq&\frac{1}{t_{0}^{\frac{d}{2}}}e^{\left|(\frac{1}{p}-\epsilon)-\frac{(1-\epsilon)}{2}\right|\left||y|^{2}-|x|^{2}\right|}e^{-\frac{(1-\epsilon)}{2}|x+y||x-y|}\\
&\leq&\frac{1}{t_{0}^{\frac{d}{2}}}e^{\left|(\frac{1}{p}-\epsilon)-\frac{(1-\epsilon)}{2}\right||x+y||x-y|-\frac{(1-\epsilon)}{2}|x+y||x-y|}\\
&=&\frac{1}{t_{0}^{\frac{d}{2}}}e^{-\alpha_{p}|x+y||x-y|}\leq C|x+y|^{d}e^{-\alpha_{p}|x+y||x-y|},
\end{eqnarray*}
where
$$\alpha_{p}=\frac{(1-\epsilon)}{2}-|(\frac{1}{p}-\epsilon)-\frac{(1-\epsilon)}{2}|.$$
Now, as $p>1$, taking $\epsilon<\frac{1}{p}$ we get that que $\alpha_{p}>0$.\\
Observe that the last expression is symmetric in $x$ and $y$ and, therefore, it suffices to prove its integrability with respect to one of them
\begin{eqnarray*}
\int_{\mathbb{R}^d} |x+y|^d e^{-\alpha_p |x+y||x-y|} dy &\leq& C + C \int_{|x-y|<1} |x|^d e^{-\alpha_p |x| |x-y|} dy \\
 && +  C \int_{|x-y|<1} |x+y |^d e^{-\alpha_p |x+y|} dy \\
&&  \leq C \int_{\mathbb{R}^d} e^{\alpha_p|v|} dv + C_d \int_0^\infty r^{2d-1} e^{-\alpha_p r} dr \leq C.
 \end{eqnarray*}
Observe that, once $p>1$ is chosen, then the operator defined using the kernel $\widetilde{K}(x,y)$ is in fact $L^q(\mathbb{R}^d)$-bounded for $1\leq q\leq \infty,$ but for the proof of the theorem it is enough the case $p=q$.
\end{itemize}
\end{enumerate}

\end{proof}

Now we will prove Theorem \ref{altrieszteo}, following the proof of Theorem 1.2 in \cite{aimarforzaniscot}, (see also Theorem 9.17 of \cite{urbina2019})

\begin{proof}

As usual, for each $x\in\mathbb{R}^d$, we write this operator as the sum of two
operators which are obtained by
splitting $\mathbb{R}^d$ into a local region,
$$B_h(x)= \{y\in \mathbb{R}^d: |y-x| < C_d m(x)\},$$ an admissible ball and its
complement $B^c_h(x)$ called the global region.  Thus,

\begin{eqnarray*}
 \overline{T}_{F,m} f(x) &=& C_d \int\limits_{|x-y|< d \,m(x)} \overline{\mathcal{K}}_{F,m} (x,y)  f(y) dy +C_d  \int\limits_{|x-y|\geq d \,m(x)} \overline{\mathcal{K}}_{F,m} (x,y)  |f(y)|dy  \\
&=&  \overline{T}_{F,m,L}f(x) +  \overline{T}_{F,m,G} f(x),
\end{eqnarray*}
where as before
$$  \overline{T}_{F,m,L}f(x)  =   \overline{T}_{F,m}(f \chi_{B_h(\cdot)})(x) $$
 is the {\em local part}
and
$$\overline{T}_{F,m,G} f(x)=  \overline{T}_{F,m} (f \chi_{B^c_h(\cdot)})(x) $$ is the {\em global  part} of $ \overline{T}_{F,m}.$

We will prove that these two operators are $\gamma_d$-weak
type $(1,1)$ and so will be $\overline{T}_{F,m}$.

In order to prove that $ \overline{T}_{F,m,L}f(x)$ is
$\gamma$-weak type $(1,1),$ we will apply Theorem 4.30 of \cite{urbina2019}. In our case,
 $$Tf(x) = \displaystyle p.v. \int_{\mathbb{R}^d} \mathcal{K}(x,y) f(y)dy $$ with
\begin{eqnarray*}
\mathcal{K}(x,y)&= &e^{|x|^2}\overline{\mathcal{K}}_{F,m}(x,y)e^{-|y|^2}\\
& =& C \int_0^1\left(\frac{-\log
r}{1-r^2}\right)^{\frac{m-2}{2}}r^{d-1} F\left(\frac{x-ry}{\sqrt{1-r^2}}\right)
\frac{e^{-\frac{|y-rx|^2}{1-r^2}}}{(1-r^2)^{\frac{d}{2}+1}}\, dr
\end{eqnarray*}
and, therefore,
\begin{eqnarray*}
\frac{\partial\mathcal{K}}{\partial y_j}(x,y)&=&2 C \int_0^1\left(\frac{-\log r}{1-r^2}\right)^{\frac{|m-2}{2}} r^{d-1}\\
 &&\hspace{-0.5cm} \times  \left[\frac{-r}{\sqrt{1-r^2}}
\frac{\partial F}{\partial y_j}\left(\frac{x-ry}{\sqrt{1-r^2}}\right)+ F\left(\frac{x-ry}{\sqrt{1-r^2}}\right)\;\frac{y_j-rx_j}{1-r^2}\right]
\frac{e^{-\frac{|y-rx|^2}{1-r^2}}}{(1-r^2)^{\frac{d}{2}+1}}\, dr.
\end{eqnarray*}

Now, we show that the hypotheses of
Theorem 4.30 of  \cite{urbina2019} are fulfilled for this operator. Thus, we prove that, in the local region $B_h(x),$ we have,
$$|\mathcal{K}(x,y)|\le \frac{C}{|x-y|^d}$$
 and
$$\left|\frac{\partial\mathcal{K}}{\partial y_j}(x,y)\right|\le \frac{C}{|x-y|^{d+1}}.$$

 There exists a constant $C>0$ such that for every $ y \in B_h(x)$
$C^{-1}\le e^{|y|^2-|x|^2}\le C$, then
$$|{\mathcal K}(x,y)|\le C |e^{-|x|^2+|y|^2}{\mathcal K}(x,y)|=C
|\overline{\mathcal{K}}_{F,m}(x,y)|$$ and
$$ \left|\frac{\partial {\mathcal K}}{\partial y_j}(x,y)\right|\le C\left|e^{-|x|^2+|y|^2}
\frac{\partial\overline{\mathcal{K}}_{F,m}}{\partial y_j}(x,y)\right|.$$
On the other hand, on $B_h(x)$, we have
$$e^{-c\frac{|y-rx|^2}{1-r^2}}= e^{-c\frac{|x-y|^2}{1-r^2}}e^{-c\frac{1-r}{1+r}|x|^2}e^{-c\frac{(x-y)\cdot
x}{1-r} }\leq C e^{-c\frac{|x-y|^2}{1-r}},$$
thus by  this
inequality and using the hypothesis on $F$, (\ref{condF3}), we have
 \begin{eqnarray*}
\left| F\left(\frac{x-ry}{\sqrt{1-r^2}}\right) \right| e^{-\frac{|y-rx|^2}{1-r^2}}\leq C_\epsilon\; e^{ \epsilon \frac{|x-ry|^2}{1-r^2}} e^{-\frac{|y-rx|^2}{1-r^2}} \leq C e^{-c\frac{|x-y|^2}{1-r}}
\end{eqnarray*}
and
\begin{eqnarray*}
\left|\nabla F\left(\frac{x-ry}{\sqrt{1-r^2}}\right)\right| e^{-\frac{|y-rx|^2}{1-r^2}} \leq C'_\epsilon \; e^{\epsilon \frac{|x-ry|^2}{1-r^2}} e^{-\frac{ |y-rx|^2}{1-r^2}} \leq C e^{-c\frac{|x-y|^2}{1-r}}
\end{eqnarray*}

\begin{eqnarray*}
|\mathcal{K}(x,y)|&\leq & C \int_0^1\left(\frac{-\log
r}{1-r^2}\right)^{\frac{m-2}{2}}\frac{e^{-c\frac{|x-y|^2}{1-r}}}{(1-r)^{\frac{d}{2}+1}}\,
dr\\
&\leq & C\left[\int_0^{\frac{1}{2}}(-\log
r)^{\frac{m-2}{2}}\, dr +
\int_{\frac{1}{2}}^1\frac{e^{-c\frac{|x-y|^2}{1-r}}}{(1-r)^{\frac{d}{2}+1}}\,
dr\right]\leq  C\left(1+\frac{1}{|x-y|^d}\right)\leq \frac{C}{|x-y|^d}
\end{eqnarray*}
and

\begin{eqnarray*}
\left|\frac{\partial\mathcal{K}}{\partial y_j}(x,y)\right| &\leq &
C\int_0^1\left(\frac{-\log r}{1-r^2}\right)^{\frac{m-2}{2}}\frac{e^{-c\frac{|x-y|^2}{1-r}}}{(1-r)^{\frac{n+3}{2}}}\,
dr\\&\leq & C\left[\int_0^{\frac{1}{2}}(-\log
r)^{\frac{|m-2}{2}}\, dr +
\int_{\frac{1}{2}}^1\frac{e^{-c\frac{|x-y|^2}{1-r}}}{(1-r)^{\frac{n+3}{2}}}\,
dr\right]\\
 &\leq & C\left(1+\frac{1}{|x-y|^{d+1}}\right)\leq \frac{C}{|x-y|^{d+1}}.
\end{eqnarray*}

From Theorem 1.2 we know that the operator $ \overline{T}_{F,m,}$ is bounded on $L^p(\gamma_d)$ for any $p>1$.
Therefore, 
$\gamma_d$-weak type $(1,1)$ of $\overline{T}_{F,m,L}$
follows, using Theorem 4.30 of \cite{urbina2019}.\\

In order to prove that $\overline{T}_{F,m,G}$ is also
$\gamma_d$-weak type $(1,1)$ we  use Forzani's generalized Gaussian maximal function,
\begin{equation}
 \mathcal{M}_{\Phi} f(x) = \sup_{0<r<1} \frac{1}{\gamma_d\left( (1+\delta)B\Big(\frac{x}{r},\frac{|x|}{r} 
(1-r)\Big)\right)} \int_{\mathbb {R}^d} \Phi\left(\frac{|x-ry|}{\sqrt{1-r^2}}\right) |f(y)| \, \gamma_d(dy), 
\end{equation}
where $\delta = \delta_{r,x} = \frac{r}{|x|(1-r)}\min\left\{\frac{1}{|x|}, \sqrt{1-r}\right\},$ see Definition 4.17 of \cite{urbina2019}, and prove that
on $\mathbb{R}^d \setminus B_h(x)$,
\begin{equation}\label{ControlGlobal}
|\overline{T}_{F,m,G} f(x)|\leq C \mathcal{M}_{\Phi}f(x),
\end{equation}
with
$\Phi(t)=e^{-c t^2}.$

\begin{eqnarray*}
|\overline{\mathcal{K}}_{F,m}(x,y)|&=&\left|\int_0^1\left(\frac{-\log
r}{1-r^2}\right)^{\frac{m-2}{2}}r^{d-1} F\left(\frac{x-ry}{\sqrt{1-r^2}}\right)
\frac{e^{-\frac{|x-ry|^2}{1-r^2}}}{(1-r^2)^{\frac{n}{2}+1}}\, dr\right|\\
&\leq &C\int_0^{\frac{3}{4}}(-\log
r)^{\frac{m-2}{2}}\frac{e^{-\frac{|x-ry|^2}{2(1-r^2)}}}{(1-r^2)^{\frac{n}{2}}}
\,
dr \\
&  &\hspace{1.5cm} +C\int_{\frac{3}{4}}^{1-\zeta/{|x|^2}}
\frac{e^{-\frac{|x-ry|^2}{2(1-r^2)}}}{(1-r^2)^{\frac{n-1}{2}}}
\ (|x|\vee (1-r^2)^{-\frac{1}{2}})\ \frac{dr}{|x|(1-r^2)^{3/2}}\\
& &\hspace{1.5cm} + C \int_{1-\zeta/{|x|^2}}^1
\frac{e^{-c\frac{|x-ry|^2}{1-r^2}}}{(1-r^2)^{\frac{n-1}{2}}}\
(|x|\vee (1-r^2)^{-\frac{1}{2}})\
\frac{e^{-\overline{c}\frac{|x-y|^2}{1-r}}}{1-r}\ dr.
\end{eqnarray*}
Hence,
$$ |\overline{\mathcal{K}}_{F,m}(x,y)|=C\left(\overline{\mathcal{K}}_{F,m}^1(x,y)
+\overline{\mathcal{K}}_{F,m}^2 (x,y)+\overline{\mathcal{K}}_{F,m}^3(x,y)\right),$$
where the inequality is obtained by annihilating the Hermite
polynomial with part of the exponential, then splitting the unit
interval of the integral into three subintervals $[0,3/4]$,
$[3/4,1-\zeta/{|x|^2}]$, and $[1-\zeta/{|x|^2},1]$ and taking into
account that on the second one $|x|\vee (1-r^2)^{-1/2}\geq |x|$, on
the third one $|x|\vee (1-r^2)^{-1/2}\ge (1-r^2)^{-1/2}$ and
$|x-ry|\ge \bar{c}|x-y|,$ and on the last two intervals  the function $-\log
r/(1-r^2)$ is bounded by a constant.

Thus, by using the definition of kernels
$\overline{\mathcal{K}}^j_{F,m}, \; j=1,2,3$; using Fubini's theorem to interchange the order
of integration on each operator ${ \overline{T}^j_{F, m, G}},\; j=1,2,3$, using the inequality
\begin{equation}
 \gamma_d \left( B \left( \frac{x}{ r} , \frac{ |x|}{ r} \, s  \right) 
\right) \le 
C  \; s^{(d-1)/2} \; \exp \left( - \frac{ |x|^2}{r^2} (1-s)^2 \right) \frac{1}{ |x|}.
\end{equation}
see Proposition 1.7 of \cite{urbina2019}, and using the definition of $\mathcal{M}_{\Phi}f$ with $\Phi (t)=e^{-ct^2}$, we get, 
\begin{eqnarray*}
{ \overline{T}^1_{F, m, G}}f(x) &=& e^{|x|^2}\int_{\mathbb{R}^d}\int_0^{\frac{3}{4}}(-\log
r)^{\frac{m-2}{2}}\frac{e^{-\frac{|x-ry|^2}{2(1-r^2)}}}{(1-r^2)^{\frac{n}{2}}}
\, dr \ |f(y)|\ \gamma_d(dy)\\
&=& \int_0^{\frac{3}{4}}(-\log
r)^{\frac{m-2}{2}}e^{|x|^2}\int_{\mathbb{R}^d}
\frac{e^{-\frac{|x-ry|^2}{2(1-r^2)}}}{(1-r^2)^{\frac{n}{2}}} |f(y)| \gamma_d(dy)\, dr\\
& \le & C \int_0^{\frac{3}{4}}(-\log
r)^{\frac{m|-2}{2}}dr \ \;\;  \mathcal {M}_{\Phi}f(x) \le  C\ \mathcal {M}_{\Phi}f(x),
\end{eqnarray*}
\begin{eqnarray*}
{ \overline{T}^2_{F, m, G}}f(x)
&=&e^{|x|^2}\int_{\mathbb{R}^d} \int_{\frac{3}{4}}^{1-\zeta/{|x|^2}}
\frac{e^{-\frac{|x-ry|^2}{2(1-r^2)}}}{(1-r^2)^{\frac{n-1}{2}}} \
(|x|\vee (1-r^2)^{-\frac{1}{2}})\ \frac{dr}{|x|(1-r^2)^{3/2}} |f(y)|\ \gamma_d(dy)\\
&=&\int_{3/4}^{1-\zeta/|x|^2}e^{|x|^2}\int_{\mathbb{R}^d}
\frac{e^{-c\frac{|x-ry|^2}{(1-r^2)}}}{(1-r^2)^{(n-1)/2}}(|x|\vee
(1-r^2)^{-1/2} )|f(y)|\gamma_d(dy) \frac{dr}{|x|(1-r^2)^{3/2}}\\
& \le& C \frac{1}{|x|}\int_{3/4}^{1-\zeta/|x|^2}\frac{ dr}{(1-r)^{3/2}}\;\;
\mathcal{ M}_{\Phi}f(x) \le C \mathcal{ M}_{\Phi}f(x),
\end{eqnarray*}
and, finally,
\begin{eqnarray*}
{ \overline{T}^3_{F, m, G}}f(x) &=&e^{|x|^2}\int_{\mathbb{R}^d} \int_{1-\zeta/{|x|^2}}^1
\frac{e^{-c\frac{|x-ry|^2}{1-r^2}}}{(1-r^2)^{\frac{n-1}{2}}}
(|x|\vee (1-r^2)^{-\frac{1}{2}})
\frac{e^{-\bar{c}\frac{|x-y|^2}{1-r}}}{1-r}\ dr\ |f(y)|\ \gamma_d(dy)\\
&=&\int_{1-\zeta/|x|^2}^{1}e^{|x|^2}\int_{\mathbb{R}^d}
\frac{e^{-c\frac{|x-ry|^2}{(1-r^2)}}}{(1-r^2)^{(n-1)/2}}(|x|\vee
(1-r^2)^{-1/2} ) \frac{e^{-\bar{c}\frac{|x-y|^2}{1-r}}}{1-r} \ |f(y)|\ \gamma_d(dy)\
dr\\& \le& \int_{1-\zeta/|x|^2}^{1}e^{|x|^2}\int_{\mathbb{R}^d}
\frac{e^{-c\frac{|x-ry|^2}{(1-r^2)}}}{(1-r^2)^{(n-1)/2}}(|x|\vee
(1-r^2)^{-1/2} )  \frac{1}{|x-y|^2}\ |f(y)| \gamma_d(dy)\ dr\\& \le& C |x|^2\int_{1-\zeta/|x|^2}^{1}dr\;\;
\mathcal{M}_{\Phi}f(x)\le C \mathcal{M}_{\Phi}f(x).
\end{eqnarray*}
So, since 
$$|\overline{T}_{F,m,G}f(x)|\le C
\sum_{j=1}^3 \overline{T}_{F,m, G}^j f(x),$$
 (\ref{ControlGlobal}) follows. Then, using Theorem 4.18 of \cite{urbina2019} (see also Theorem 1.1 of \cite{aimarforzaniscot}) we get the $\gamma_d$-weak type $(1,1)$
inequality for $\overline{T}_{F,m,G}.$\end{proof}

In an forthcoming paper \cite{NavPinUr}, following \cite{DalSco}, we prove that the general alternative Gaussian singular integrals $ \overline{T}_{F,m} $ are also  continuous on  Gaussian variable Lebesgue spaces under a condition of regularity on $p(\cdot).$\\

\bibliographystyle{amsplain}

\end{document}